\documentclass[12pt]{amsart}

\usepackage{graphicx}
\usepackage{tikz}
\usetikzlibrary{cd}

\newtheorem{theorem}{Theorem}
\newtheorem{thmalpha}{Theorem}
\newtheorem{lemma}{Lemma}
\newtheorem{proposition}{Proposition}
\newtheorem{corollary}{Corollary}
\theoremstyle{definition}
\newtheorem{definition}{Definition}
\theoremstyle{remark}

\newcommand{\R}{\mathbb{R}}
\newcommand{\Z}{\mathbb{Z}}
\newcommand{\N}{\mathbb{N}}

\DeclareMathOperator{\diam}{diam}
\DeclareMathOperator{\conv}{Conv}

\DeclareMathOperator{\Cech}{C\check{e}ch}
\DeclareMathOperator{\CechMod}{\textbf{C\v ech}}
\DeclareMathOperator{\VRips}{VR}

\DeclareMathOperator{\rank}{rank}
\DeclareMathOperator{\dgm}{dgm}

\DeclareMathOperator{\kmax}{kmax}

\begin{document}


\title{The Persistence Landscapes of Affine Fractals}
\author{Michael J. Catanzaro}
\address{Department of Mathematics, Iowa State University, 396 Carver Hall, Ames, IA 50011}
\email{mjcatanz@iastate.edu}
\author{Lee Przybylski}
\address{Department of Mathematics, Iowa State University, 396 Carver Hall, Ames, IA 50011}
\email{leep@iastate.edu}
\author{Eric S. Weber}
\address{Department of Mathematics, Iowa State University, 396 Carver Hall, Ames, IA 50011}
\email{esweber@iastate.edu}
\date{\today}
\begin{abstract}
We develop a method for calculating the persistence landscapes of affine fractals using the parameters of the corresponding transformations.  Given an iterated function system of affine transformations that satisfies a certain compatibility condition, we prove that there exists an affine transformation acting on the space of persistence landscapes which intertwines the action of the iterated function system.  This latter affine transformation is a strict contraction and its unique fixed point is the persistence landscape of the affine fractal.  We present several examples of the theory as well as confirm the main results through simulations.
\end{abstract}
\maketitle



\section{Introduction}


Affine fractals are the invariant set of an iterated function system (IFS) consisting of affine transformations acting on Euclidean space.  Well-known examples of such fractals are Cantor sets and Mandelbrojt sets.  Affine fractals, as subsets of Euclidean space, possess topological properties that can be extracted through the methods of algebraic topology.  In particular, these subsets of Euclidean space can be associated to a persistence landscape \cite{bubenik2015statistical} which is a sequence of functions that encode geometric properties of the set based on Euclidean distances.  These distances give rise to a family of homology groups derived from a filtration of complexes.  The homology groups in turn produce a persistence module from which the persistence landscapes are defined.

Interest in studying fractals using the tools of algebraic topology has occurred recently.  In \cite{Robins_thesis}, it was shown that persistence homology can be used to distinguish fractals of the same Hausdorff dimension.  In \cite{MATE2014252}, the authors describe a relationship between the Hausdorff dimension of fractals and the persistence intervals of Betti numbers.  

Our main result (Theorem \ref{thm : Op_WSI_IFS}) concerns the calculation of the persistence diagrams and landscapes of affine fractals.  We prove that, under a certain compatibility condition, there exists an affine transformation $\mathcal{L}$ which is defined by the parameters of the IFS.  This transformation $\mathcal{L}$ acts on the space of persistence landscapes.  Moreover, it is a strict contraction and its unique fixed point is the persistence landscape of the affine fractal.  Consequently, the persistence landscape of the fractal can be computed via the limiting process of repeated applications of $\mathcal{L}$ to any initial input.  We also prove, under an additional assumption on the IFS, that $\mathcal{L}$ intertwines the action of the iterated function system  (Theorem \ref{thm : IFS_Op_Invariance}).   

\subsection{Affine Fractals}

A fractal, for our purposes, is a set which has a self-similarity property.  The middle-third Cantor set is the canonical example of a self-similar set.  Fractals are commonly studied objects in many contexts.  Cantor sets, in particular, appear in the context of analysis \cite{cantor1884puissance,St99a}, number theory \cite{Cas59a,Sch62a}, probability  \cite{lyons2017probability,Jor06,Byars2021sampling}, geometry \cite{Pol88a,peres1998self}, and harmonic analysis \cite{JP98,St00,RavStr16a}.  

In this paper, we consider specifically the class of affine fractals, which are generated by iterated function systems consisting of affine transformations.  By this we mean that the fractal is the invariant set for the iterated function system.
\begin{definition}
Suppose $\Psi =\{\psi_1,...,\psi_N\}$ is a set of maps acting on a metric space $(X,d)$.  We say that $A \subset X$ is invariant for $\Psi$ if $A = \cup_{i=1}^{N} \psi_{i}(A)$.
\end{definition}
For the maps $\Psi$, we denote the compositions (i.e. iterations) of the maps by:
\[\Psi(A) = \bigcup_{i=1}^N \psi_i(A),\:\: \Psi^p(A) = \Psi(\Psi^{p-1}(A)).\]

In \cite{hutchinson1981}, Hutchinson laid out the main relationship between fractals and iterated function systems (IFS); this relationship is the foundation of our results.  Recall that $\psi:X\to X$ is Lipschitz if for all $x,y\in X$, there exists $C>0$ such that 
\[d(\psi(x),\psi(y))\le Cd(x,y).\]
The Lipschitz constant of $\psi$ is the infimum of all such $C$.  We say that $\psi$ is a contraction if it has Lipschitz constant less than 1.

\begin{thmalpha}\label{thm : Hutch1981}
Let $X = (X,d)$ be a complete metric space and $\Psi = \{\psi_1,...,\psi_N\}$ a finite set of contraction maps on $X$.  Then there exists a unique closed bounded set $A$ such that
\[A = \bigcup_{i=1}^N \psi_i(A).\]
Furthermore, $A$ is compact and is the closure of the set of fixed points of finite compositions of members of $\Psi$. Moreover, for a closed bounded $K$, $\Psi^p(K)\to A$ in the Hausdorff metric.
\end{thmalpha}

Recall that the Hausdorff distance between two sets $A,B \subset X$ is given by
\[ d_{H}(A,B) = \max \left\{ \sup_{x \in A} \inf_{y \in B} d(x,y) , \sup_{y \in B} \inf_{x \in A} d(x,y) \right\}. \]



For us, the maps $\Psi$ consist of affine transformations acting on $\R^d$.  Moreover, we assume that the linear part of the maps are scalars which are common to all of the maps.  Therefore, our maps have the form $\psi_{j}(\vec{x}) = c(\vec{x} + \vec{b}_{j})$, where $c \in (0,1)$.
The following are examples of affine fractals in our class:
\begin{enumerate}
\item the classical middle-third Cantor set in $\R$; 
\item the Sierpinski gasket in $\R^2$;
\item the Sierpinksi carpet in $\R^2$;
\item the Menger sponge (or Sierpinski cube) in $\R^3$.
\end{enumerate}

The Cantor set is the invariant set for the IFS with generators
\[ \psi_{0}(x)= \frac{x}{3} ; \quad \psi_{1}(x) = \frac{x + 2}{3}. \]
The Sierpinski carpet is the invariant set for the IFS with generators
\[ \psi_{0}(x,y)= \left(\frac{x}{3}, \frac{y}{3} \right) ; \quad \psi_{1}(x) = \left(\frac{x+2}{3}, \frac{y}{3} \right); \quad \psi_{2}(x) = \left(\frac{x}{3}, \frac{y+2}{3} \right); \quad \psi_{3}(x) = \left(\frac{x+2}{3}, \frac{y+2}{3} \right). \]

For an IFS $\Psi$ on $\R^d$ and a nonempty $S_0 \subset \R^d$, we define the sequence $S_{n+1} = \Psi(S_{n})$.  We will typically consider $S_{0}$ which consists of finitely many points in $\R^d$, and therefore by Hutchinson's theorem, $\{ S_{n} \}$ converges in the Hausdorff metric to the fractal generated by $\Psi$.  Indeed, we will show that choosing $S_{0}$ to be the extreme points of the convex hull of the fractal is ideal in establishing our algorithm for calculating the persistence landscapes of the fractal.


One of our main results is to prove that for a fixed affine IFS $\Psi$ that satisfies a certain compatibility condition and appropriate initialization $S_{0}$, there exists an affine transformation $\mathcal{L}$ acting on the space of persistence landscapes such that the following diagram commutes for every $n \in \mathbb{N}$:

\begin{equation}\label{eq : com_dgm0}
\begin{tikzcd}
  S_{n} \arrow[r, "\Psi"] \arrow[d,"\gamma"]  & 
  S_{n+1} \arrow[d,"\gamma"]\\
  \mathbf{f}_{n} \arrow[r,"\mathcal{L}"] & \mathbf{f}_{n+1}
\end{tikzcd}
\end{equation}
Here $\mathbf{f}_{n}$ is the persistence landscape of $S_{n}$ and $\gamma$ associates to $S_{n}$ its persistence landscape.  We will show that $\mathcal{L}$ is a strict contraction on the set of persistence landscapes, and so it possesses a unique fixed point.  That fixed point will be the persistence landscape for the fractal generated by $\Psi$. Consequently, the persistence landscape $\mathbf{f}$ of the fractal is obtained by
\begin{equation}
    \mathbf{f} = \lim_{n \to \infty} \mathcal{L}^{n} \mathbf{f}_{0}
\end{equation}
for any initialization $\mathbf{f}_{0}$.

\subsection{Persistence Landscapes}
Persistent homology is a relatively new approach to 
studying topological spaces. In the context of data science,
persistent homology can be applied to a data set to complement 
traditional statistical approaches by studying the geometry 
of the data. 
We employ the tools of persistent
homology, including persistence landscapes, to analyze affine
fractals.

Persistent homology typically begins with a set of points, equipped with a pairwise
notion of distance. We place a metric ball of radius $r$ around each
point, and increase $r$. We are interested in the topological
properties of the union of these balls \textit{as a function
of $r$}.  Typical properties of interest include connectedness, 
loops or holes, and voids of the union, and importantly the radii at which these appear or disappear.
While this type of information may seem crude, a surprising amount of insight about the underlying set of points can be extracted in this way. 

The data of changing topological properties is conveniently 
summarized in what is known as a {\em persistence diagram}, a multiset
of points in the plane.
If we focus on loops of the union, then each point $(b,d)$ in the 
persistence diagram represents a loop, whose two coordinates
correspond to the radius when the hole is formed ($r=b$) and when it
gets filled in ($r=d$).  The persistence diagram provides a
multiscale summary, encoding geometric and topological features of the set~\cite{carlsson_persistence_2005}.  
 
 

Unfortunately however, barcodes do not posses a vector space structure, so
quantitative analysis and precise comparison can be
difficult~\cite{munch_probabilistic_2015}. To remedy this, we map the barcodes
to some feature space (a Banach space in our case) using a well-studied feature
map known as a {\em persistence landscape}. The mapping from barcodes to landscapes is
reversible, so this vectorization scheme loses no information~\cite{bubenik2015statistical}. Persistence landscapes have been used to study protein binding~\cite{kovacev-nikolic_using_2016}, phase transitions~\cite{donato_persistent_2016}, audio signals~\cite{liu_applying_2016}, and microstructures in materials science~\cite{dlotko_topological_2016}.



\section{Persistent Homology}

In this section, we briefly review some standard facts from
algebraic topology and persistent homology as well as establish our notation. Excellent resources for (simplicial) homology can
be found in~\cite{hatcher, munkres_eta}, and~\cite{edelsbrunner2010computational, carlsson_persistence_2005, perea_brief_2018} provide a good introduction to persistent homology.

\subsection{Simplicial Complexes}

%
For a simplicial complex $K$, the \textit{$p$-skeleton of $K$}, denoted by $K^{(p)}$, is the subcomplex consisting of all simplices of dimension less
than or equal to $p$. The set of all $p$-simplices is denoted $K_p$. 
Thus, the set of vertices can be written as $K_0$. Recall that if $\sigma = [u_0, u_1, \ldots, u_p]$ is a $p$-simplex, then every point $x \in \sigma$ can be expressed as a convex combination $x = \sum_i \alpha_i u_i$, where 
$0 \leq \alpha_i \leq 1$ and $\sum_i \alpha_i = 1$.



Given two simplicial complexes $K$ and $L$, 
we say that $\varphi:K_0 \to L_0$ is a \emph{vertex map} if for any $p$-simplex $[x_0,...,x_p]$ in $K$, $[\varphi(x_0),...,\varphi(x_p)]$ is 
a simplex in $L$. Thus vertex maps send the vertices of simplices in $K$ to 
simplices in $L$. We do not require $\varphi$ to be injective, so $\dim[\varphi(x_0), ...,\varphi(x_p)]\le p$ with a strict inequality if $\varphi(x_j) = \varphi(x_k)$ for some $j,k\in\{0,1,...,p\}$.  Given a vertex map $\varphi:K_0\to L_0$, we can extend it to a map $f:K \to L$ by 
\[
f(x) = \sum_{j=0}^p\alpha_j\varphi(x_j) \,.
\]
where $x = \sum_{j=0}^p \alpha_j x_j$.  In this way, we say that $f$ is the simplicial map induced by $\varphi$.  


The first step in our goal of computing topological properties of affine fractals
will be to construct their C\v ech complexes.
If $X\subset\R^d$, for any $\varepsilon>0$ we can define the C\v ech complex to be
\begin{equation} \label{eqn:Cech}
    \Cech(X,\varepsilon) = \left\{\sigma\subseteq X\:\bigg| \bigcap_{x\in\sigma}\overline{B(x,\varepsilon/2)}\not=\emptyset\right\} \, ,
\end{equation}
where $B(x, r)$ is the ball of radius $r$ centered at $x$.  A consequence of the Nerve Theorem \cite{BorsukKarol1948Otio,edelsbrunner2010computational} is that for a finite set $X\subset\R^d$, $\{ x : d(x,X) \leq \varepsilon/2 \}$ is homotopy equivalent to $\Cech(X,\varepsilon)$.

There is another popular variant in persistent homology for associating
a topological space to a set, known as the
Vietoris-Rips complex. The Vietoris-Rips complex for $X\subset\R^d$ and $\varepsilon>0$ is defined by
\begin{equation} \label{eqn:VR}
    \VRips(X,\varepsilon) = \left\{\sigma\subseteq X \big| \max_{x,y\in\sigma}|x-y|\le \varepsilon\right\} \, .
\end{equation}
In comparing Eqs.~\eqref{eqn:Cech} and~\eqref{eqn:VR}, we see that the 1-simplices
in $\Cech(X,\varepsilon)$ are the same as those in $\VRips(X,\varepsilon)$,
but it is not necessary that $\Cech(X,\varepsilon) = \VRips(X,\varepsilon)$.
Furthermore, verifying the existence of a point in the intersection of
Eq.~\eqref{eqn:Cech} often requires much more work than verifying
the \textit{pairwise} condition of Eq.~\eqref{eqn:VR}. For this reason,
together with recent advances in computational efficiency in software~\cite{bauer2021ripser}, applications of persistent homology
tend to rely on Vietoris-Rips complexes.
The two complexes are related by a well-known result~\cite[Thm. 2.5]{de_silva_coverage_2007}.

Our focus will be on the $\Cech$  complex of affine fractals and their approximations.


\subsection{A Review of Homology}
Given a simplicial complex $K$, an abelian group $G$, and a 
non-negative integer $p \geq 0$, we define the \textit{group of
$p$-chains with coefficients in $G$} to be formal $G$-linear combinations
of $p$-simplices of $K$ and denote it $C_p(K;G)$. A typical element
of $C_p(K;G)$ is a finite formal sum of the form $\sum_i g_i \sigma_i$, 
where $\sigma_i \in K_p$ and $g_i \in G$. The \textit{differential}
(or \textit{boundary}) of a $p$-simplex $\sigma = [u_0,u_1,\ldots,u_p]$ is 
\begin{equation}
\label{eqn:differential}
    \partial_p (\sigma) = \sum_{j=0}^p (-1)^j [u_0, u_1,\ldots,\hat{u}_j,
    \ldots,u_p] \in C_{p-1}(K;G) \, ,
\end{equation}
where $[u_0, \ldots, \hat{u}_j, \ldots u_p]$ is the $(p-1)$-simplex
obtained by omitting $\hat{u}_j$ from $\sigma$. Extending $\partial_p$
to a $G$-linear homomorphism to all of $C_p(K;G)$ gives $\partial_p:C_p(K;G) \to C_{p-1}(K;G)$.
This gives rise to the \textit{simplicial chain complex of $K$ with 
coefficients in $G$}
\begin{equation}
    \cdots \overset{\partial_{p+2}}\longrightarrow C_{p+1}(K)\overset{\partial_{p+1}}\longrightarrow C_{p}(K)\overset{\partial_{p}}\longrightarrow C_{p-1}(K)\overset{\partial_{p-1}}\longrightarrow \cdots \overset{\partial_{1}}\longrightarrow C_0(K)\overset{\partial_{0}}\longrightarrow 0 \, ,
\end{equation}
with the essential property that any two successive compositions 
equal the trivial map: $\partial_{j}\partial_{j+1} = 0$ for all $j \geq
0$. Hence, $\mathrm{im}(\partial_{j+1}) \subset \ker(\partial_j)$. 
We denote the \textit{$p$-cycles} by $Z_p(K) = \ker(\partial_p) \subset C_p(K)$,
and the \textit{$p$-boundaries} by $B_p(K) = \mathrm{im}(\partial_{p+1}) 
\subset C_p(K)$.


\begin{definition}
The \textbf{p-th simplicial homology group} of a simplicial complex $K$ is
\[H_p(K) = H_p(K;G):= \ker(\partial_p)/\mathrm{im}(\partial_{p+1}) = Z_p(K)/B_p(K) \, .\]
We let $H_*(K) = \bigoplus_p H_p(K)$ denote the collection of homology
groups for all dimensions $p$.
\end{definition}

Moving forward, we assume $G = \Z_2$. This choice of coefficients is
common in persistent homology and simplifies many of the 
computations, e.g., the factors of $(-1)^j$ appearing in
Eq.~\eqref{eqn:differential} vanish.


For simplicial complexes $K$ and $L$, $f:K\to L$ is a \textit{simplicial map} if $f$ is continuous and $f$ maps each simplex of $K$ linearly onto a simplex of $L$. 
Define a homomorphism $f_\#:C_p(K)\to C_p(L)$ by first defining
\begin{equation}\label{eq : induced_hom}
    f_\#([x_0,...,x_p]) =
    \begin{cases}
    [f(x_0),...,f(x_p)]  &\text{if}\:\: f(x_0),...,f(x_p)\:\:\text{are distinct}\\
    0  &\text{otherwise,}
    \end{cases} 
\end{equation}
and extend the homomorphism to the rest of $C_p(K)$ linearly.  A standard
fact in algebraic topology is that $f_\#$ further induces a
map on homology $f_*:H_p(K) \to 
H_p(L)$ for every $p$. Furthermore, if $f$ is a simplicial homeomorphism, then
$f_*:H_*(K) \to H_*(L)$ is an isomorphism~\cite{munkres_eta}.




\begin{lemma}  \label{cor : sim_hom_isomorphism}
  Let $X\subset\R^d$ be a finite point cloud.  Let $\varphi:\R^d\to \R^d$ be a similitude with scaling constant $c>0$.  Let $\varepsilon > 0$, $L = \Cech(X,\varepsilon)$, and  let $\tilde{L} = \Cech(\varphi(X), c\varepsilon)$.  Then  $\varphi|_{L_0} : L_0 \to \tilde{L}_0$ is a vertex map and induces a simplicial homeomorphism $f$ between $L$ and $\tilde{L}$. 
  Thus, $f$ induces an isomorphism $f_*:H_*(L) \to H_*(\tilde L)$.
\end{lemma}

A particularly nice feature of homology groups is that the homology 
group of a space is isomorphic to the direct sum of the 
homology groups of the path components~\cite[Prop. 2.6]{hatcher}. 
This directly leads to the following lemma.


\begin{lemma}\label{lemma : dir_sum_hom}
Let $X\subset \R^d$ be a finite subset and $\{\varphi_j\}_{j=1}^n$ be a collection of similitudes on $\R^d$.  Let $c_j$ equal the scaling constant of the similitude $\varphi_j$.  
Define
\[\delta :=\min_{1\le j\not= k\le n} d(\varphi(X_j),\varphi(X_k)).\]
Then for all dimensions $p\ge 0$, and all $\varepsilon <\delta$
\[H_p\Cech(\cup_{j} \varphi_j(X),\varepsilon)\cong \bigoplus_{j=1}^n H_p\Cech(X,c_j^{-1}\varepsilon).\]
\end{lemma}

Lemma~\ref{lemma : dir_sum_hom} implies that given a finite point cloud and a collection of similitudes, so long as the images of the point cloud under those similitudes are 
sufficiently far apart, the homology group resulting from the union of those images is easily related to the homology groups resulting from the original point cloud.  




Another tool we will use for the computation of homology groups is known
as the \textit{Mayer-Vietoris sequence}. 
Suppose $K$ is a simplicial complex with subcomplexes $K_1$ and $K_2$ such that $K = K_1\cup K_2$.  The Mayer-Vietoris sequence is the long
exact sequence
\begin{equation}\label{eq : Mayer-Vietoris}
\begin{split}
     \cdots\longrightarrow& H_p(K_1\cap K_2)\longrightarrow H_p(K_1)\oplus H_p(K_2)\longrightarrow H_p(K)\longrightarrow H_{p-1}(K_1\cap K_2)\longrightarrow\cdots\\
     & \cdots \longrightarrow H_0(K_1\cap K_2)\longrightarrow H_0(K_1)\oplus H_0(K_2)\longrightarrow H_0(K)\longrightarrow 0.
\end{split}
\end{equation}
For a detailed explanation of the maps in the sequence and a 
proof of its exactness, see~\cite[Chapter 3]{munkres_eta}.  

\subsection{Persistent Homology}
A \textit{simplicial filtration} $K^\bullet$ of a simplicial complex
$K$ is a collection of subcomplexes $\{K^i\}_{i \in I}$, for
an indexing set $I \subset \R$ such that whenever $s \leq t$,
\begin{equation}\label{eq : filtration}
K^s\subseteq K^t\subseteq K.
\end{equation}
The inclusion map $K^s\to K^t$ induces a homomorphism 
$H_*(K^s)\to H_*(K^t)$. The collection of groups $\{H_p(K^s)\}_{s\in I}$, together with the collection of homomorphisms $\{H_p(K^s) 
\to H_p(K^t)\}_{s \leq t}$ form the 
\textit{$p^{\mathrm{th}}$-persistent homology of $K^{\bullet}$}. 

Understanding the structure of persistent homology is difficult
as presented. Instead, we turn to a generalization of this structure
known as a persistence module.
A \textit{persistence module} is a collection of $\Z_2$-vector spaces
$\mathbb{V} = \{V_t\}_{t\in I}$, and a collection of linear transformations
between these spaces $\mathbb{M} = \{M^{s,t}:V_s\to V_t\}_{s\le t}$ that satisfy
\begin{equation}\label{eq : module_comp}
M^{s,t}\circ M^{r,s} = M^{r,t}\:\:\:\text{for}\:\: r\le s\le t.
\end{equation}
We will often refer to the persistence module $(\mathbb{V},\mathbb{M})$ as simply $\mathbb{M}$. 
By setting $V_s = H_p(K^s)$ and $M^{s,t}:H_p(K^s) \to H_p(K^t)$ as the homomorphism induced by inclusion, we
see that every simplicial filtration $K^{\bullet}$ gives rise to a persistence module.

If $X\subset\R^d$, then the collection of simplicial complexes 
\[\CechMod(X) = \{\Cech(X,s)\}_{s\in\overline{\R}} \]
form a simplicial filtration of $\Cech(X,\infty)$, the full
simplicial complex on $X$.
For $0\le s\le t$, the inclusion mapping $\Cech(X,s)\to\Cech(X,t)$ induces a homomorphism $H_p\Cech(X,s)\to H_p\Cech(X,t)$.  We let $H_p\CechMod(X)$ denote the persistence module resulting from the filtration of the C\v ech complex of $X$.  
We adopt the convention that 
\[\Cech(X,s) = \VRips(X,s) = \emptyset\]
for $s<0$.  

A persistence diagram is a multiset of birth and death times derived from a decomposition of a persistence module.  Interval modules form the basic building blocks of persistence modules.  Interval modules are indecomposable.  In \cite{botnan2019decomposition}, it was shown that the persistence module $(\mathbb{V},\mathbb{M})$ can be decomposed into interval modules as long as each $V_t$ is finite dimensional.  Moreover, an application of~\cite[Thm. 1]{AzumayaGoro1950CaSt}, implies this decomposition is unique, up to reordering.  Thus if we have a persistence module $\mathbb{M}$ with the interval module decomposition
\[\mathbb{M} = \bigoplus_{\lambda}\mathbb{J}[b_\lambda,d_\lambda),\]
we define the persistence diagram of $\mathbb{M}$ to be the multiset
\[\dgm\mathbb{M} = \{(b_\lambda,d_\lambda)\}_{\lambda}\cup\Delta,\]
where $\Delta$ is the set of diagonal points $\{(x,x)|x\in\R\}$ each counted with infinite multiplicity.  

We will denote a persistence diagram as $\mathcal{D} = (D,\mu)$, where $D =\{(b_\lambda,d_\lambda)\}_{\lambda}$ is the collection of distinct birth and death times and $\mu:D\to \N_0$ is defined $\mu(b_\lambda,d_\lambda)$ as the multiplicity of the birth death time $(b_\lambda,d_\lambda)$.  

In order to ensure that the persistence modules we consider have a decomposition into interval modules, we require that $\mathbb{M}$ is q-tame \cite{chazal2013structure}.  
\begin{definition}
The persistence module $(\mathbb{V},\mathbb{M})$ indexed over $I\subset\R$ is said to be \textbf{q-tame} if
\[\rank M^{s,t}<\infty\:\:\text{whenever}\:\: s<t.\]
\end{definition}
Assuming a persistence module is q-tame not only guarantees us a well-defined persistence diagram, but also stability with respect to the bottleneck distance \cite{chazal:inria-00292566}.  A prior stability result presented by Cohen-Steiner et al. in \cite{Cohen-SteinerDavid2007SoPD} which required a more restrictive tameness condition. 

Persistence modules resulting from a C\v ech filtration built on a finite point cloud $X$ are $q$-tame. Moreover, when $X\subset\R^d$ is compact,
$H_p\CechMod(X)$ is $q$-tame~\cite{chazal2013persistence} for all $p$. As
a result, we obtain the following:
\begin{proposition}
For any affine iterated function system, the invariant set $A$, and any finite approximation of $A$, possess a well-defined persistence diagram.
\end{proposition}


While persistence diagrams are an effective representation of a persistence module, they are not conducive to statistical analysis.  Persistence landscapes address this issue by embedding persistence diagrams into a Banach space. 
Given a persistence module $\mathbb{M}$ that is $q$-tame, its \textit{persistence landscape}
is a function $\mathbf{f}:\N\times \R \to\overline{\R}$ by
\[\mathbf{f}(n,t) := f^{(n)}(t) = \sup\{m\ge 0|\:\rank M^{t-m,t+m}\ge n\} \, ,\]
where $\overline{\R}$ denotes the extended real numbers, $[-\infty, \infty]$.
We have that all persistence landscapes are elements of 
\[L^\infty(\N\times \R):= \left\{\mathbf{g} = \{g^{(n)}\}_{n=1}^\infty \subset L^\infty(\R)\bigg| \sup_{n\in\N} \|g^{(n)}\|_{L^\infty(\R)} <\infty\right\}.\]
We compute the distance between two persistence landscapes using the standard norm on $L^\infty(\N\times\R)$ which is defined as
\[\|\mathbf{g}\|_{L^\infty(\N\times\R)} = \sup_{n\in\N}\|g^{(n)}\|_{L^\infty(\R)}.\]

There is an alternative definition given in \cite{BubenikPeter2020TPLa} that allows us to relate the persistence landscape of $\mathbb{M}$ to its persistence diagram.  If the persistence module $\mathbb{M}$ is represented as a persistence diagram $D = \{(a_i,b_i)\}_{i\in I}$, then we can define the 
functions
\begin{equation}\label{eq : hatfunc_def}
    \tau_{(a,b)}(t) = \max(0,\min(a+t,b-t)),
\end{equation}
and for all $k\in\N$, $t\in\R$,
\begin{equation}\label{eq : kmaxdef}
f^{(k)}(t)=\text{kmax}\{\tau_{(a_i,b_i)}(t)\}_{i\in I}.
\end{equation}
We use $\text{kmax}$ to denote the $k$th largest element of a set.  Note that since $D$ is a multiset, certain birth death pairs $(a_i,b_i)$ can appear more than once.  

It is common for many stability results to write the distance between two persistence landscapes in terms of their corresponding persistence modules.  If $\mathbf{f} = \{f^{(k)}\}_{k=1}^\infty$ is the persistence landscape obtained from $\mathbb{M}$ and $\mathbf{g}=\{g^{(k)}\}_{k=1}^\infty$ is the persistence landscape obtained from $\mathbb{M}'$, we define the persistence landscape distance between $\mathbb{M}$ and $\mathbb{M}'$ to be
\begin{equation}
\Lambda_\infty(\mathbb{M},\mathbb{M}') := \| \mathbf{f} - \mathbf{g} \|_{L^\infty(\N\times\R)}.
\end{equation}

By combining the stability results from \cite{bubenik2015statistical,chazal2013persistence,chazal:inria-00292566}, we obtain the following:

\begin{thmalpha} \label{Th:stability}
Suppose $X$ is a metric space, and $Y,Z \subset X$ have well-defined persistence diagrams.  Then for all $p \geq 0$,
\[ \Lambda_\infty(H_p\CechMod(Y),H_p\CechMod(Z))\le d_{H}(Y,Z)\, . \]
\end{thmalpha}

\begin{corollary}
Let $(X,d)$ be a metric space and suppose $S, T \subset X$ such that their
C\v{e}ch complexes have well-defined persistence diagrams. Then 
\[ \| \mathbf{f} - \mathbf{g} \|_{L^\infty(\N\times\R)} \leq d_{H}(S,T), \]
where $\mathbf{f}$ and $\mathbf{g}$ are the persistence landscapes of $S$ and $T$, respectively.
\end{corollary} \label{C:stability}

\section{The Persistence Landscapes of Affine Fractals} \label{sec:affine}

\subsection{The Persistence Landscape of the Cantor Set} \label{ssec:cantor3}
We begin with calculating the persistence landscape of the middle-third Cantor set $\mathcal{C}$.  This calculation illustrates the main ideas of our construction, while also highlighting some of the technical requirements on the IFS.  Recall that $\mathcal{C}$ is generated by the IFS given by the collection of contraction maps $\Psi = \{\psi_0,\psi_2\}$, where
\[\psi_0(x) = \frac{x}{3},\:\:\psi_2(x) = \frac{x}{3}+\frac{2}{3}.\]
For convenience, we suppose that $\Psi$ acts on the metric space $X= [0,1]$ with the standard metric.  We define the sequence of approximations to $\mathcal{C}$ as follows:
\begin{equation} \label{eq:cantor3}
S_0 = \{0,1\},\:\:S_n =\Psi(S_{n-1}) = \psi_0(S_{n-1}) \cup \psi_2(S_{n-1})\:\text{for}\:\:n\in\N.
\end{equation}
We also define
\[ C_{0} = [0,1],\:\:C_{n} = \Psi(C_{n-1}). \]
The length of each closed interval in $C_n$ is $1/3^n$.  Note that since the points in $S_n$ are equal to the end points in the disjoint closed intervals that make up $C_n$, and we have $S_{n} \subset \mathcal{C}\subset C_n$ for all $n\in \mathbb{N}$, we find
\[ d_H(S_n, \mathcal{C}) \leq d_H(S_n, C_n) \le \frac{1}{2 \cdot 3^n}.\]
As a consequence of this and Lemma \ref{C:stability}, we obtain the following convergence result.

\begin{theorem}\label{thm : C_landscapes}
Let $\{ \mathbf{f}_{n} \}_{n=1}^{\infty}$ be the sequence of persistence landscapes generated from the sequence of point clouds $\{S_n\}_{n=1}^\infty$ as in Equation \eqref{eq:cantor3}, and let $\mathbf{f}$ be the persistence landscape of $\mathcal{C}$. Then, we have that
\[ \lim_{n \to \infty} \| \mathbf{f}_{n} - \mathbf{f} \|_{L^\infty(\N\times\R)} = 0. \]
\end{theorem}

Knowing that the limit exists, we would still like to have a formula for the persistence landscape of $\mathcal{C}$.  To help us determine this formula, we will use a two step approach.  First, we will find an affine operator $\mathcal{L}:L^\infty(\N\times\R)\to L^\infty(\N\times\R)$ such that for all $n$ sufficiently large, $\mathcal{L}$ maps the persistence landscape of $S_n$ to the persistence landscape of $S_{n+1}$.  We will then show that the persistence landscape of $\mathcal{C}$ is the fixed point of $\mathcal{L}$ and use $\mathcal{L}$ to compute the fixed point.


To describe persistence landscapes, we will use the hat functions defined in Equation~\eqref{eq : hatfunc_def}.  Recall that persistence landscapes $\mathbf{h}\in L^\infty(\N\times\R)$ are defined by letting $\mathbf{h} = \{h^{(j)}\}_{j=1}^\infty$, where $h^{(j)}$ equals the $j$th largest hat function supported on an interval whose end points equal one of the birth-death pairs from the persistence diagram.  We also note that the hat functions scale nicely so that for $t\in\R$
\[\tau_{(a,b)}(t) = (b-a)\tau_{(0,1)}((b-a)^{-1}(t-a)).\]
We define $\mathcal{L}$ for each $\mathbf{g}\in L^\infty(\N\times\R)$ by 
$\mathcal{L}\mathbf{g} = \mathbf{h} = \{h^{(j)}\}_{j=1}^\infty$ where
\begin{align} \label{eq:L-cantor3}
h^{(1)}(x) &= \tau_{(0,1)}(x),\\
h^{(2)}(x) &= \frac{1}{3}g^{(1)}(3x) , \notag \\
h^{(2k+1)}(x)&=h^{(2k +2)}(x) = \frac{1}{3}g^{(k)}(3x),\:\:k\in\N \, . \notag
\end{align}
We adopt the convention that the maximum death time will be equal to the diameter of the invariant set, which in the case of $\mathcal{C}$ is 1.  Thus when computing persistence landscape of $\mathcal{C}$, the first function in the sequence will be $\tau_{(0,1)}$.  We can see that $\mathcal{L}$ is an affine contraction.  Indeed, if we define $\mathcal{T}:L^\infty(\N\times\R)\to L^\infty(\N\times\R)$ by $\mathcal{T}\mathbf{g} = \mathbf{h}$, where
\[\begin{split}
h^{(1)}(x) &= 0,\\
h^{(2)}(x) &= \frac{1}{3}g^{(1)}(3x)\\
h^{(2k+1)}(x)&=h^{(2k +2)}(x) = \frac{1}{3}g^{(k)}(3x) \, ,\\
\end{split}\]
then we 
see that $\mathcal{T}$ 
is linear with $\|\mathcal{T}\| = \frac{1}{3}$. Further, for all $\mathbf{f}\in L^\infty(\N\times\R)$, $\mathcal{L}\mathbf{f} = \mathbf{v} + \mathcal{T}\mathbf{f}$, where $\mathbf{v} = \{\tau_{(0,1)}, 0,0,0,...\}$.  
\begin{theorem}\label{thm : cantorset_op} 
Let $\{ \mathbf{f}_{n} \}_{n=1}^{\infty}$ be the sequence of persistence landscape generated from the sequence of point clouds $\{S_n\}_{n=1}^\infty$ as in Equation \eqref{eq:cantor3}, and let $\mathbf{f}$ be the persistence landscape of $\mathcal{C}$. Let $\mathcal{L}$ be given by Equation \eqref{eq:L-cantor3}.  Then, for all $n\in\N$, $\mathcal{L}\mathbf{f}_n = \mathbf{f}_{n+1}$.
\end{theorem}

In other words, the commutative diagram in Equation \eqref{eq : com_dgm0} holds.

\begin{proof}
For $n\in\N$, we denote $\mathbf{f}_n = \{f_n^{(j)}\}_{j=1}^\infty$.  It easy to check for $n = 1$ using direct computation.  Since $S_1 = \{0,1\}$ and $S_2 = \{0,1/3,2/3,1\}$ we see that $f_2^{(1)} = \tau_{(0,1)}$, $f_2^{(j)} = \tau_{(0,3^{-1})}$ for $2\le j\le 4$, and $f_2^{(j)} = 0$ for all $j> 4$.  We can also see that 
\[\mathbf{f}_1 = \{\tau_{(0,1)}, 0,0,0,0,...\}.\]
By the definition of $\mathcal{L}$, we see that $\mathcal{L}\mathbf{f}_1 = \mathbf{f}_2$.   In general, for $n\in\N$, let 
\[L_n = \phi_0(S_n),\:\:R_n = \phi_2(S_n).\]
For all $n\in\N$, we see that
\[S_{n+1} = L_n\cup R_n.\]
Since $S_n\subset[0,1]$, we know that $L_n\cap R_n = \emptyset$.  Moreover,
\begin{equation}\label{eq : LRCdist}
d(L_n,R_n) = \min_{x\in L_n,y\in R_n} |x-y| = \frac{1}{3}.
\end{equation}
Let $\mathcal{D}_n= (D_n,\mu_n)$ denote the persistence diagram of $H_0\CechMod(S_n)$.  Note that since $S_n$ has $2^n$ distinct points, if $D_n = \{(0,d_j)\}_{j=1}^k$, then $\sum_{j=1}^k\mu_n(0,d_j) = 2^n$.  We assume without loss of generality that $d_j< d_{j+1}$ for $j\in\{1,...,k\}$.  This means $d_k = 1$, and for $j< k$, $d_j<1$.

\begin{figure}[b]
    \centering
    \includegraphics[scale = 0.65]{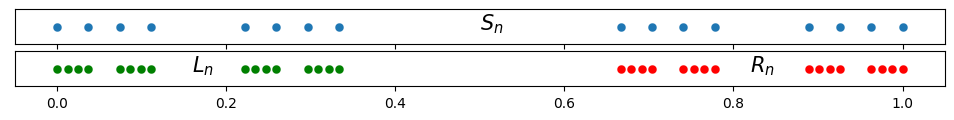}
    \caption{$S_n$ divided into the components $L_n$ and $R_n$ for $n = 3$.}
    \label{fig : cantorLnRn}
\end{figure}

We claim that for $j<k$, $(0,d_j)\in D_n$ implies that $(0,\frac{d_j}{3})\in D_{n+1}$ with $\mu_{n+1}(0,d_j/3) = 2\mu_n(0,d_j)$.  To help us do this, we first prove that for $0<\varepsilon<1$, the diagram
\begin{equation}\label{eq : com_dgm1}
\begin{tikzcd}
  H_0\Cech(S_n,0) \arrow[r, "M^{0,\varepsilon}"]\arrow[d,"\phi_{0*}^0"]  & H_0\Cech(S_n,\varepsilon)\arrow[d,"\phi_{0*}^\varepsilon"]\\
  H_0\Cech(L_n,0) \arrow[r,"P^{0,\varepsilon}"] & H_0\Cech(L_n,\tfrac{\varepsilon}{3})
\end{tikzcd}
\end{equation}
commutes, where $M^{0,\varepsilon}$ and $P^{0,\varepsilon}$ are homomorphisms induced by the inclusions $\Cech(S_n,0)\to\Cech(S_n,\varepsilon)$ and $\Cech(L_n,0)\to\Cech(L_n,\frac{\varepsilon}{3})$ respectively, and $\phi_{0*}^0$ and $\phi_{0*}^\varepsilon$ are isomporphisms induced by $\phi_0$ as in Corollary~\ref{cor : sim_hom_isomorphism}.  Indeed, for $\gamma\in H_0\Cech(S_n,0)$, we may use coset notation to write $\gamma = [x]B_0(\Cech(S_n,0))$ for some $x\in S_n$.  Thus
\[P^{0,\varepsilon}\phi_{0*}^0\gamma = P^{0,\varepsilon}[\tfrac{x}{3}]B_0\Cech(L_n,0) = [\tfrac{x}{3}]B_0\Cech(L_n,\tfrac{\varepsilon}{3}).\]
On the other hand,
\[\phi_{0*}^\varepsilon M^{0,\varepsilon}\gamma = \phi_{0*}[x]B_0\Cech(S_n,\varepsilon)=[\tfrac{x}{3}]B_0\Cech(L_n,\tfrac{\varepsilon}{3}).\]
Therefore $\phi_{0*}^\varepsilon M^{0,\varepsilon}= P^{0,\varepsilon}\phi_{0*}^0$, which implies  Diagram~\eqref{eq : com_dgm1} commutes.

To prove the first claim, suppose $(0,d_j)\in D_n$ with $j<k$, then there are $\mu_n(0,d_j)$ distinct classes $\gamma\in H_0\Cech(S_n,0)$ that die at $d_j$.  By assumption, there exists $\gamma\in H_0\Cech(S_n,0)$ such that $\gamma\in\ker M^{0,d_j}$, but $\gamma\notin\ker M^{0,t}$, for $t < d_j$.  Let $\gamma^- = \phi_{0*}^0\gamma$.  Since the diagram in \eqref{eq : com_dgm1} commutes, $\gamma^-\in\ker P^{0,\varepsilon}$ if and only if $\gamma\in\ker M^{0,\varepsilon}$.  Thus $\gamma^{-}\in \ker P^{0,d_j}$, but $\gamma^-\notin\ker P^{0,t}$ for $t<d_j$.  

For $s\le t$, let $Q^{s,t}:H_0\Cech(R_n,\frac{s}{3})\to H_0\Cech(R_n,\frac{t}{3})$ be the homomorphism induced by the natural inclusion. 
By replacing $L_n$ with $R_n$, and $\tfrac{x}{3}$ with $\frac{x+2}{3}$ in the argument above, we can see that there also exists $\gamma^+\in H_0\Cech(R_n,0)$ such that $\gamma^{+}\in \ker Q^{0,d_j}$, but $\gamma^+\notin\ker Q^{0,t}$ for $t<d_j$.  
Since $\Cech(L_n,\varepsilon)$ is not path-connected to
$\Cech(R_n,\varepsilon)$ in $\Cech(S_{n+1},\varepsilon)$ for $\varepsilon<\frac{1}{3}$, 
it follows 
that 
\[H_0\Cech(S_{n+1},\varepsilon)\cong  H_0\Cech(L_n,\varepsilon)\oplus H_0\Cech(R_n,\varepsilon).\]
For $0\le s\le t<\frac{1}{3}$, the homomorphism $F^{s,t}:H_0\Cech(S_{n+1},\frac{s}{3})\to H_0\Cech(S_{n+1},\frac{t}{3})$ induced by the obvious inclusion 
satisfies
\[\ker F^{s,t}\cong \ker P^{s,t}\oplus\ker Q^{s,t}.\]
Thus for every class $\gamma\in H_0\Cech(S_n,0)$ with death time $d_j<1$, there exist 2 distinct classes in $H_0\Cech(S_{n+1})$ with death time $\frac{d_j}{3}$, which implies $(0,\frac{d_j}{3})\in D_{n+1}$ with $\mu_{n+1}(0,\frac{d_j}{3}) = 2\mu_n(0,d_j)$.  This is equivalent to what we claimed.

It follows from our claim that
\[\{(0,\tfrac{d_j}{3})\}_{j =1}^{k-1}\subset D_{n+1}\]
and
\[\sum_{j=1}^{k-1}\mu_{n+1}(0,\tfrac{d_j}{3}) = 2\sum_{j=1}^{k-1}\mu_n(0,d_j)= 2(2^{n} -1) = 2^{n+1}-2.\]
By~(\ref{eq : LRCdist}) we also see that $(0,1/3)\in D_{n+1}$.  
Since we are working with the zero-dimensional homology of a C\v{e}ch complex,
we also have $(0,1)\in D_{n+1}$, since the space is non-empty. 
This accounts for all $2^{n+1}$ elements of $D_{n+1}$. For convenience, we reindex $D_n$ and $D_{n+1}$ repeating elements according to their multiplicity and ordering them by decreasing death times so that 
\[D_{n} = \{(0,1), (0,d_{(2)}), (0,d_{(3)}), ...,(0,d_{(2^n)})\},\]
and 
\[D_{n+1} = \{(0,1), (0,1/3), (0,d_{(2)}/3), (0,d_{(2)}/3), ...,(0,d_{(2^n)}/3),(0,d_{(2^n)}/3)\} .\]
Applying the definition in (\ref{eq : kmaxdef}), we see that $\mathbf{f}_n = \{f_n^{(j)}\}_{j=1}^\infty$, is defined by
\[f_n^{(j)} = \tau_{(0,d_{(j)})}, \text{ for } j\in\{1,2,...,2^{n}\}, \:\text{and}\:f_n^{(j)}=0\: \text{for}\: j> 2^{n}.\]
Since $d_{(1)} = 1$ and $d_{(2)} = 1/3$, we easily check that $\mathbf{f}_{n+1} = \{f_{n+1}^{(j)}\}_{j=1}^\infty$ satisfies 
\[\begin{split}
f_{n+1}^{(1)} &= \tau_{(0,1)},\\
f_{n+1}^{(2)}(t) &= \tau_{(0,1/3)}(t) = \frac{1}{3}f_{n}^{(1)}(3t)\\
f_{n+1}^{(2j+1)}(t) = f_{n+1}^{(2j+2)}(t) &= \tau_{(0,d_{(j)}/3)}(t) = \frac{1}{3}f_j^{(n)}(3t)\:\text{for } j\in\N .\\ 
\end{split}\]
Therefore $\mathcal{L}\mathbf{f}_n = \mathbf{f}_{n+1}$.
\end{proof}
Since $\mathcal{L}$ is Lipschitz with constant $\|\mathcal{T}\| = \frac{1}{3}<1$, we see that $\mathcal{L}$ has a unique fixed point, and that unique fixed point is $\mathbf{f} = \lim_{n\to\infty}\mathbf{f}_n$, which is the persistence landscape function of $\mathcal{C}$.  The explicit formula for $\mathbf{f}$ can be found using $\mathcal{L}$ by repeatedly applying $\mathcal{L}$ to any vector in $L^\infty(\N\times\R)$.  We find that the persistence landscape of $\mathcal{C}$ is $\mathbf{f} = \{f^{(j)}\}_{j=0}^\infty$, where
\begin{equation}\label{eq : Cantor_PL}
f^{(j)} = 
\begin{cases}
\tau_{(0,1)} &\text{if}\:\:j = 1\\
\tau_{(0,3^{-k})} &\text{if}\: \:2^{k-1}< j\le 2^{k},\:\:k\in\N\, .
\end{cases}
\end{equation}
We illustrate this landscape in Figure~\ref{fig : cantor-3} (produced by pyscapes \cite{angeloro_pyscapes_2020}).

We see from the illustration that the persistence landscape exhibits its own version of self-similarity.  This is a reflection of the fact that the fractal contains several scaled copies of itself.  Indeed, since scaling a subset of Euclidean space results in a proportional scaling of its persistence landscape, we should expect the persistence landscape of a fractal to contain a subsequence which is a scaled copy of itself.  The number of scaled copies, which corresponds to the number of generators of the IFS, is also reflected as a multiplicity in the persistence landscape.


\begin{figure}
    \includegraphics[width=12cm]{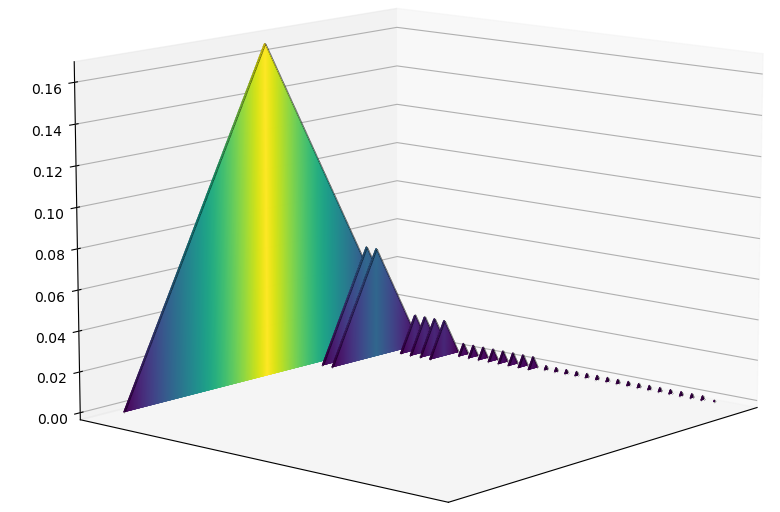}
    \caption{Graph of the functions $f_{2}, \dots, f_{33}$ from the persistence landscape of the middle third Cantor set $\mathcal{C}$. }
    \label{fig : cantor-3}
\end{figure}

\subsection{Affine Fractals with Well-Separated Images and Extreme Points}
The proof of Theorem~\ref{thm :  cantorset_op} suggests that a more general result exists for an IFS satisfying certain properties.  The two main ingredients that enable our calculations in the proof include 1) a judicious choice for the initial approximation $S_{0}$ and 2) a compatibility condition of the images of the maps in $\Psi$.  We refer to this condition as \emph{well-separated images}, and formalize this condition in Definition \ref{D:wsi}.  We first consider the choice of $S_{0}$.  Before proceeding, we will need to introduce some notation and definitions.  Unless stated otherwise, we assume $\Psi = \{\psi_{j}\}_{j=1}^N$ is an IFS consisting only of similitudes on $\R^d$ with the form
\begin{equation}\label{eq : simple_sim}
    \psi_j(\vec{x}) = c(\vec{x}+\vec{b}_j),\hspace{1cm} \vec{b}_j\in\R^d,\:c\in(0,1).
\end{equation}

We let $A$ denote the invariant set of $\Psi$.  For a set $B \subset \R^d$, we let Conv($B$) denote the convex hull of $B$.  We let $E_{B}$ denote the set of extreme points of Conv($B$).  As we shall see, for the affine fractal $A$, choosing $S_{0}=E_{A}$ is a good choice of initialization of our algorithm.  This corresponds to the choice we made for the middle third Cantor set $\mathcal{C}$; see also the example in Section \ref{ssec : mod1/5} for further evidence of this assertion.

Since the maps in $\Psi$ are contractions, each map has a unique fixed element.  Indeed, it is easy to calculate that for $\psi_{j}$, the fixed point is $\vec{x}_{j} = \left(\frac{c}{1-c}\right) \vec{b}_{j}$.  We let $F_{A}$ denote the set of these fixed points.  Theorem \ref{thm : Hutch1981} guarantees that $F_{A} \subset A$.  The following result tells us that we can easily find $E_{A}$.

\begin{lemma}\label{lemma : extreme_fixed}
Suppose $A\subset \R^d$ is the invariant set for some IFS $\Psi = \{\psi_j\}_{j=1}^N$ consisting of similitudes of the form in Equation~\eqref{eq : simple_sim}.  Then $E_A \subseteq F_A$.  
\end{lemma}

\begin{proof}
Let $K = \conv(A)$.  Clearly, $F_A\subset K$.  Assume $F_A = \{\vec{x}_{j}\}_{j=1}^N$with $\vec{x}_{j} = \psi_j(\vec{x}_{j})$. We first observe that for $j\in \{1,...,N\}$, since $\vec{x}_{j} = c(\vec{x}_{j} + \vec{b}_{j})$, we know that $c\vec{b}_{j} = (1-c)\vec{x}_{j}.$  This implies that for $j\not=k$, 
\[\psi_k(\vec{x}_{j}) = c\vec{x}_{j} + c\vec{b}_{k} = c\vec{x}_{j}+(1-c) \vec{x}_{k}\in\conv(F_A).\] 
If $\vec{y}\in\conv(F_A)$, then for some $t_1,...,t_N\ge 0, \sum_{j=1}^{N} t_j = 1$ we have
\[y = \sum_{j=1}^{N} t_j \vec{x}_{j},\]
and for any $\psi_k\in\Psi$,
\[\begin{split}
    \psi_k( \vec{y}) &= c\left(\sum_{j=1}^N t_j\vec{x}_{j} + \vec{b}_{k}\right)= \sum_{j=1}^N ct_j\vec{x}_{j} + c\vec{b}_{k}= \sum_{j=1}^N ct_j\vec{x}_{j} + \sum_{j=1}^Nt_j c\vec{b}_{k}\\
    &= \sum_{j=1}^N t_j c(\vec{x}_{j} + \vec{b}_{k}) = \sum_{j=1}^N t_j\psi_k(\vec{x}_{j}).
    \end{split}\]
Since $\psi_k(\vec{x}_{j})\in \conv(F_A)$ for all $j\in \{1,...,N\}$, this implies that $\psi_k(\conv(F_A)) \subseteq \conv(F_A)$.  Since $k$ was arbitrary, this implies that the union of these images, $\Psi(\conv(F_A))\subseteq\conv(F_A)$.  From Theorem~\ref{thm : Hutch1981}, we have 
\begin{equation}\label{eq : haus_conv}
A = \lim_{p\to\infty}\Psi^p(\conv(F_A))
\end{equation}
 in the Hausdorff metric.  We claim that $A\subseteq\conv(F_A)$.  Indeed, choose $\vec{x}\in A$.  For all $n\in \N$, it follows from Equation \eqref{eq : haus_conv} that there exists $p(n)$ such that for some $\vec{y}_n\in \Psi^{p(n)}(\conv(F_A))$,
 \[|\vec{y}_{n} - \vec{x} |\le \frac{1}{n}\]
Since $\Psi^{p(n)}(\conv(F_A)) \subset \conv(F_A)$ for all $n\in \N$, we have $\{\vec{y}_n\}_{n=1}^\infty\subset\conv(F_A)$.  Since $\conv(F_A)$ is a closed set and $\lim_{n\to\infty}\vec{y}_n = \vec{x}$, this implies $\vec{x}\in\conv(F_A)$, which proves the claim.   Thus we have the sequence of containment:
\[F_A\subseteq A\subseteq\conv(A)\subseteq\conv(F_A).\]
By definition, this implies $\conv(F_A) = \conv(A)$.  By the Krein-Milman Theorem, it follows that
\[\conv(F_A) = \conv(A) = \conv(E_A).\]
For any $\vec{y}\in E_A$, $\vec{y}\in\conv(F_A)$, it follows that $\vec{y} =\sum_{j=1}^N t_j\vec{x}_{j}$, where $\sum_{j=1}^N t_j = 1$.  Since $\vec{y}$ is an extreme point, for some $k\in\{1,...,N\}$, $\vec{y} = \vec{x}_k\in F_A$.  Therefore $E_A \subseteq F_A$.
\end{proof}


One key property of $\mathcal{C}$ that was used in the proof of Theorem~\ref{thm : cantorset_op} was that at each scale, the set could be partitioned into a left and right set.  The two halves were a significant distance away from each other, and each half was a scaled down version of the previous scale.  This property can be described in terms of the IFS, and because of its usefulness, we will define it formally.

\begin{definition} \label{D:wsi}
Let $\Psi = \{\psi_j\}_{j=1}^N$ be an IFS with invariant set $A$.  We say that $\Psi$ has \textbf{well-separated images} (or satisfies the well-separated condition) if 
\begin{equation} \label{eq:well-separated}
\min_{1\le j\not=k  \le N} d(\psi_j(A), \psi_k(A))\ge \max_{1\le j\le N} \diam\psi_j(A).
\end{equation}
\end{definition}
This definition may apply to any IFS, not only those of the form given in Equation~\eqref{eq : simple_sim}.  Note that on the left hand side of the inequality, we have the usual Euclidean distance, not the Hausdorff distance.  




\subsection{Main Results}

Using the well-separated condition and the ideas in Subsection \ref{ssec:cantor3}, we are now ready to elucidate the relationship between an IFS and the persistence landscape of its invariant set in more generality.  Our main focus is on IFS with well-separated images having the form in Equation~\eqref{eq : simple_sim}, but many of the results below do not require these assumptions.  We will use the same two step approach that we used with $\mathcal{C}$.  We first identify a contraction on $L^\infty(\N\times\R)$ that has a fixed point equal to the persistence landscape of interest, then use the operator to find a formula for the persistence landscape.

Theorem~\ref{thm : Hutch1981} implies that if $\Psi$ is any IFS consisting of contractions then iteratively applying $\Psi$ to any compact set $K\subset\R^d$ creates a sequence of compact sets,  $\{S_n\}_{n=1}^\infty$, that converges to the invariant set $A$ in the Hausdorff metric.  Therefore, as a consequence of Theorem \ref{Th:stability}:


\begin{theorem}\label{thm : invariant_set_PLconv}
Let $\Psi$ be an IFS on $\R^d$ consisting of contractions with invariant set $A$.  Let $K\subset\R^d$ be a compact set, and define the sequence of compact sets $\{S_n\}_{n=1}^\infty$ by 
\[S_1 = K,\:\: S_n = \Psi(S_{n-1})\:\:\text{for}\:\: n>1.\]
Then for any $p \geq 0$
\[\lim_{n\to\infty}\Lambda_\infty(H_p\CechMod(S_n), H_p\CechMod(A)) = 0.\]
\end{theorem}

We remark that the statement of Theorem~\ref{thm : invariant_set_PLconv} only mentions the C\v ech filtration, which applies to any dimension of homology.  Also note that the hypothesis makes no assumptions on $\Psi$ except that it consists of contractions.

Having established that there is a sequence of persistence landscapes that converge to the persistence landscape of the invariant set, we now seek a contraction on $L^\infty(\N\times\R)$ whose fixed point is the landscape of interest.  Just like we did above with $\mathcal{C}$ above, we will approximate $A$ by a finite set $F$ and compare the persistence landscapes of $F$ and $\Psi(F)$ to determine the operator.  

\begin{proposition}\label{prop : FiniteHomApprox}
Let $\Psi =\{\psi_j\}_{j=1}^N$ be an IFS consisting of similitudes all with scaling constants $c_j\in (0,1)$.  Let $A\subset\R^d$ be the invariant set of $\Psi$. For any $\varepsilon >0$ and any $p \geq 0$, there exists a finite set $F\subseteq A$ such that the following hold
\begin{enumerate} 
    \item[(a)] $d_H(A,F) \le\varepsilon$
    \item[(b)] $d_H(A,\Psi(F))\le\varepsilon$
    \item[(c)] $\Lambda_\infty(H_p\CechMod(A),H_p\CechMod(F)) \le\varepsilon$
    \item[(d)] $\Lambda_\infty(H_p\CechMod(A),H_p\CechMod(\Psi(F)) \le\varepsilon$.
\end{enumerate}
\end{proposition}
\begin{proof}
Choose $\varepsilon>0$.  Since $A$ is compact, there exists $F := \{x_j\}_{j=1}^n\subseteq A$ such that $A\subseteq \bigcup_{j=1}^n B(x_j,\varepsilon)$.  Since $A\subseteq F_\varepsilon$ and $F\subset A_\varepsilon$, \emph{(a)} is satisfied.  Since $F\subseteq A$, we have $\Psi(F)\subseteq A\subset A_\varepsilon$.  To show the other containment, choose $z\in A$.  Since $\Psi(A) = A$, there exists $k\in\{1,...,N\}$, $w\in A$ such that $\psi_k(w) = z$.  Thus for some $x_j\in F$, $|x_j-w| <\varepsilon$, which implies
\[|\psi_k(x_j)-z| = |\psi_k(x_j) - \psi_k(w)| = c|x_j-w|<c\varepsilon<\varepsilon.\]
Thus $A\subseteq (\Psi(F))_\varepsilon$.  Thus \emph{(b)} is satisfied. Applying Theorem \ref{Th:stability}, we know for any compact set $K\subset\R^d$, 
\begin{equation}\label{eq : landscape_hausdorf_bound}
    \Lambda_\infty(H_p\CechMod(A),H_p\CechMod(K))\le d_{H}(A,K)
\end{equation}
In light of Equation~\eqref{eq : landscape_hausdorf_bound}, \emph{(c)} follows from \emph{(a)}.  Similarly, \emph{(d)} follows from \emph{(b)}.
\end{proof}

\begin{definition}
For a disconnected set $X\subset\R^d$, we say that $X$ is $\varepsilon$-connected if it cannot be expressed as a union of two non-empty sets $Y,Z\subset X$ such that $d(Y,Z)>\varepsilon$.  We say that $Y\subseteq X$ is an $\varepsilon$-component of $X$ if $Y$ is $\varepsilon$-connected and 
\[d(Y,X\backslash Y)>\varepsilon.\]
We let $C(X,\varepsilon)$ denote the number of distinct $\varepsilon$-components of $X$.
\end{definition}

Note that $C(X, \varepsilon)$ equals the number of connected components of $\cup_{x\in X}\overline{B(X,\varepsilon)}$, which is precisely the rank of 
$H_0\CechMod(X, \varepsilon)$. 
Clearly this number is non-increasing with respect to $\varepsilon$.  

If we assume $\Psi$ is an IFS of similitudes with well-separated images and the scaling constant for each $\psi_j \in \Psi$ equals $c\in (0,1)$, then we can define a sequence of distances 
\[\delta_1\ge \delta_2\ge \dots\ge \delta_{N-1}\ge \delta_N\]
by letting $\delta_1 = \diam(A)$, and for $k\in\{2,...,N\}$,
\begin{equation}\label{eq : resolutions_def}
    \delta_k := \inf\{\varepsilon>0|C(A,\varepsilon)\le k-1\}.
\end{equation}
We will make use of the fact that when $\Psi$ has well-separated images,
\[\delta_N = \min_{1\le j\not= k\le N} d(\psi_j(A),\psi_k(A)).\]
This also means that
\begin{equation}\label{eq : delta_j_def}
    \delta_N \ge c\delta_1 = \diam\psi_j(A),\:\:j\in\{1,...,N\}.
\end{equation}



\begin{proposition}\label{prop : eps_comp_estimates}
Suppose $A\subset\R^d$ is compact, and $F\subseteq A$ is finite with $d_H(F,A) \le \alpha$.  Then for all $\varepsilon>2\alpha$, we have
\[C(A,\varepsilon)\le C(F,\varepsilon)\le C(A,\varepsilon -2\alpha).\]
\end{proposition}

\begin{proof}
Suppose first that $S\subseteq A$ is an $\varepsilon$-component of $A$.  This means $\overline{S}_{\varepsilon/2}$ is connected and $d(S,A\backslash S)>\varepsilon$.  Let $S_0 = F\cap S$.  Since $\alpha <\varepsilon$, we know that $S_0\not=\emptyset$.  We have $S_0\subseteq S$ and $F\backslash S_0\subseteq A\backslash S$.  Indeed, if $x\in F\backslash S_0$, then $x\in F$ and $x\notin F\cap S$, which implies $x\notin S$.  Therefore $x\in A\backslash S$.  With this claim, we have established that
\[d(S_0,F\backslash S_0)\ge d(S,A\backslash S) >\varepsilon.\]
Since $S_0$ is nonempty and compact, it contains at least one distinct $\varepsilon$-component of $F$.  Since $S$ was an arbitrary $\varepsilon$-component of $A$, this implies
\[C(F,\varepsilon)\ge C(A,\varepsilon).\]
For the other inequality, suppose now that $S_0\subseteq F$ is an $\varepsilon$-component of $F$.  Let $S= A\cap (S_0)_{\alpha}$.  Choose $x\in S$ and $y\in A\backslash S$.  By assumption, there exists $w\in S_0$ with $|x-w| <\alpha$.  Since $y\notin S$, we have for all $v\in S_0$, $|y-v|\ge\alpha$.  Since $A\subset F_\alpha$, there exists $z\in F$, with $|y-z|<\alpha$, which implies $z\notin S_0$.  Therefore $|z-w|\ge\varepsilon$.  Hence
\[\varepsilon\le |w-z|\le |w-x|+|x-y|+|y-w|.\]
Therefore
\[\varepsilon-2\alpha <\varepsilon -|w-x|-|y-w|\le |x-y|.\]
This implies that $d(S,A\backslash S)>\varepsilon-2\alpha$.  From this bound, we see that if we partition $F$ into its $n$ distinct $\varepsilon$-components $S_0,S_1,...,S_n$, then by letting $S'_j = A\cap(S_j)_\alpha$, we see that $A = \bigcup_{j=0}^n S'_j$.  We also have that each $S'_j$ contains at least one distinct $(\varepsilon-2\alpha)$-component of $A$.  Thus 
\[C(F,\varepsilon)\le C(A,\varepsilon-2\alpha). \qedhere \]
\end{proof}

Using the distances as defined in Equation \eqref{eq : resolutions_def}, we define $\mathcal{L}:L^\infty(\N\times\R)\to L^\infty(\N\times\R)$ by $\mathcal{L}\mathbf{g} = \mathbf{h}$, where 
\begin{equation}\label{eq : well-sep_op_def}
    \begin{split}
        h^{(1)} = \tau_{(0,\delta_1)},\:h^{(2)} = \tau_{(0,\delta_{2})} \, ,\ldots, \, 
        h^{({N-1})} = \tau_{(0,\delta_{N-1})}, h^{(N)} = \tau_{(0,\delta_N)}, \\
        h^{(kN+1)}= h^{(kN+2)} = \cdots = h^{(kN+N-1)}=h^{(kN+N)} =
        cg^{(k+1)}(c^{-1}x),\:\:\text{for}\:\:k\in\N \, .
    \end{split}
\end{equation}

\begin{theorem}\label{thm : Op_WSI_IFS}
Let $\Psi$ be a IFS on $\R^d$ consisting of $N$ similitudes each with scaling constant $c\in(0,1)$.  Let $A$ be the invariant set of $\Psi$ and let $\mathbf{f}\in L^\infty(\N\times\R)$ denote the persistence landscape of resulting from $H_0\CechMod(A)$.  If $\Psi$ has well-separated images, then $\mathcal{L}:L^\infty(\N\times\R)\to L^\infty(\N\times\R)$ as defined in Equation~\eqref{eq : well-sep_op_def}, satisfies $\mathcal{L}\mathbf{f} = \mathbf{f}$.
\end{theorem}


\begin{proof} 
Let $\mathbf{f}\in L^\infty(\N\times\R)$ denote the persistence landscape resulting from $H_0\CechMod(A)$.    By Proposition~\ref{prop : FiniteHomApprox}, for all $n\in \N$, there exists a finite subset $F_n\subseteq A$ such that $d_H(A,\Psi(F_n))\le \frac{1}{n}$ with
\[\Lambda_\infty(H_0\CechMod(A),H_0\CechMod(F_n)) \le\frac{1}{n}\:\:\text{and}\:\:\Lambda_\infty(H\CechMod(A),H_0\CechMod(\Psi(F_n)) \le\frac{1}{n}.\]
Choose $\varepsilon >0$.  Let $\mathbf{g}_n\in L^\infty(\N\times\R)$ denote the persistence landscape from $H_0\CechMod(F_n)$ and $\mathbf{h}_n\in L^\infty(\N\times\R)$ denote the persistence landscape from $H_0\CechMod(\Psi(F_n))$.  Using the triangle inequality, for all $n\in\N$, we have
\begin{equation}\label{eq : Lf_minus_f}
    \|\mathcal{L}\mathbf{f}-\mathbf{f}\|_{L^\infty(\N\times\R)} \le \|\mathcal{L}\mathbf{f}-\mathcal{L}\mathbf{g}_n\|_{L^\infty(\N\times\R)}+\|\mathcal{L}\mathbf{g}_n-\mathbf{h}_n\|_{L^\infty(\N\times\R)}+\|\mathbf{h}_n-\mathbf{f}\|_{L^\infty(\N\times\R)}.
\end{equation}
Looking at the definition of $\mathcal{L}$ in Equation~\eqref{eq : well-sep_op_def}, we see that $\mathcal{L}$ is a contraction with Lipschitz constant $c$.  Thus 
\[\|\mathcal{L}\mathbf{f}-\mathcal{L}\mathbf{g}_n\|_{L^\infty(\N\times\R)}\le c\|\mathbf{f}-\mathbf{g}_n\|_{L^\infty(\N\times\R)}\le c\Lambda_\infty(H_0\CechMod(A),H_0\CechMod(F_n))\le \frac{c}{n}.\]
Similarly, 
\[\|\mathbf{h}_n-\mathbf{f}\|_{L^\infty(\N\times\R)} = \Lambda_\infty(H_0\CechMod(A),H_0\CechMod(\Psi(F_n)) \le\frac{1}{n}.\]
Now our goal is to bound $\|\mathcal{L}\mathbf{g}_n-\mathbf{h}_n\|_{L^\infty(\N\times\R)}$.  For $s\le t$, let $M_n^{s,t}$ denote the homomorphism induced by the inclusion mapping $\Cech(F_n,s)\to \Cech(F_n,t)$. For $s\le t$, let $P_n^{s,t}$ denote the homomorphism induced by the inclusion mapping $\Cech(\Psi(F_n),s)\to \Cech(\Psi(F_n),t)$.  We have $\mathbf{g}_n = \left\{g_{n}^{(j)}\right\}_{j=1}^\infty$, which implies $\mathbf{\hat{g}}_n = \mathcal{L}\mathbf{g}_n$, where 
\begin{equation}\label{eq : ghat}
    \begin{split}
        \hat{g}_n^{(1)} = \tau_{(0,\delta_1)},\:\hat{g}^{(2)}_n = \tau_{(0,\delta_{2})},...,\hat{g}^{(N-1)}_n = \tau_{(0,\delta_{N-1})}, \hat{g}^{(N)}_{n} &= \tau_{(0,\delta_N)}\\
        \hat{g}^{(kN+1)}_{n}= \hat{g}^{(kN+2)}_{n} = ...= \hat{g}^{(kN+N-1)}_{n}=\hat{g}^{(kN+N)}_{n} &= cg_n^{(k+1)}(c^{-1}x),\:\:\text{for}\:\:k\in\N
    \end{split}
\end{equation}
Note that we use the convention $\Cech(X,r) = \emptyset$ for any $X\subseteq \R^d$ and $r<0$, which implies that $\rank i_*^{r,t} = 0 $ for $r<0$, where $i_*:H_0\Cech(X,r)\to H_0\Cech(X,t)$ is the homomorphism induced by inclusion.  This is why all landscape functions from C\v ech filtrations have non-negative support.  

Let $\eta_1 = \diam(F)$ and define $\eta_2,...,\eta_N$ by 
\[\eta_k := \inf\{\varepsilon>0|C(\Psi(F_n),\varepsilon)\le k-1\}\]
Since $d_H(A,\Psi(F_n))<\frac{1}{n}$, if we assume $n\ge \frac{4}{\delta_N}$, then we have
\[\frac{2}{n}\le \frac{\delta_N}{2}<\delta_N.\]
We know that for all $\varepsilon < \delta_N$, we have $C(A,\varepsilon) >k-1$.  It follows from Proposition~\ref{prop : eps_comp_estimates} that
\begin{equation}
    \eta_k  = \inf\{\varepsilon >0|C(\Psi(F_n),\varepsilon)\le k-1\}
    \ge \inf\{\varepsilon >0|C(A,\varepsilon)\le k-1\} = \delta_k.
\end{equation}
On the other hand, we have 
\begin{equation}
    \begin{split}
        \eta_k &= \inf\{\varepsilon >0|C(\Psi(F_n),\varepsilon)\le k-1\}\le \inf\{\varepsilon>0|C(A,\varepsilon-\tfrac{2}{n})\le k-1\}\\
        &= \inf\{\varepsilon +\tfrac{2}{n} >0|C(A,\varepsilon)\le k-1\}
        =\frac{2}{n} +\inf\{\varepsilon>0|C(A,\varepsilon)\le k-1\}\\
        &=\delta_k+\frac{2}{n}.
    \end{split}
\end{equation}
Thus we have  $|\eta_k-\delta_k| \le \frac{2}{n}$.  We also know that $c\diam\Psi(F_n)\le \eta_N$ since $\Psi$ has well separated images.
By Lemma~\ref{lemma : dir_sum_hom}, we have for $t<c\eta_N$, 
\[H_0\Cech(\Psi(F_n),t)\cong \bigoplus_{j=1}^N H_0\Cech(F_n,c^{-1}t).\]
From this isomorphism, we can see that for  $s\le t\le \eta_N$, the homomorphism $P_n^{s,t}:H_0\Cech(\Psi(F_n),s)\to H_0\Cech(\Psi(F_n),t)$ has the same rank as
\[M_+^{c^{-1}s,c^{-1}t}:\bigoplus_{j=1}^N H_0\Cech(F_n,c^{-1}s)\to \bigoplus_{j=1}^N H_0\Cech(F_n,c^{-1}t),\]
where $M_+^{s,t} =\bigoplus_{j=1}^NM^{s,t}$.  Thus for $s\le t\le \eta_1$,
\[\rank P^{s,t} = N\rank M^{c^{-1}s,c^{-1}t}.\]

By definition, we have for all $j\in\N,\:t\ge 0$,
\[h_n^{(j)}(t) = \sup\{m\ge 0 |\rank P_n^{t-m,t+m}\ge j\}.\]
If $t+m\ge \eta_N$, then $\Cech(\Psi(F_n),t+m)$ has at most, $N-1$ connected components, which implies $\rank P_n^{t-m,t+m}<N$, so for $j = kN + l$ for $k\in\N$, $l\in \{1,2,...,N\}$, we have
\begin{equation}\label{eq : hjs}
\begin{split}
    h_n^{(j)}(t) &= \sup\{m\ge 0 |\rank P_n^{t-m,t+m}\ge j\}\\
    &=\sup\{m\ge0|\rank M_n^{c^{-1}(t-m),c^{-1}(t+m)}\ge j/N\} \\
    &=\sup\{m\ge0|\rank M_n^{c^{-1}(t-m),c^{-1}(t+m)}\ge k+\tfrac{l}{N}\}\\
    &=\sup\{m\ge0|\rank M_n^{c^{-1}(t-m),c^{-1}(t+m)}\ge k+1\}
    \end{split}
\end{equation}

On the other hand, for any $k\in \N$, 
\begin{equation}\label{eq : gks}
    \begin{split}
        cg_n^{(k)}(c^{-1}t) &= c\sup\{m\ge 0|\rank M_n^{c^{-1}t-m,c^{-1}t+m}\ge k\}\\
        &=\sup\{cm\ge 0|\rank M_n^{c^{-1}t-m,c^{-1}t+m}\ge k\}\\
        &=\sup\{m\ge 0|\rank M_n^{c^{-1}(t-m),c^{-1}(t+m)}\ge k\}\\
    \end{split}
\end{equation}
Combining Equation~\eqref{eq : hjs} and Equation~\eqref{eq : gks}, we see that for $j =kN+l$, $k\in\N$, $l\in\{1,2,...,N\}$, we have
\[h_n^{(j)}(t) = cg_{n}^{(k+1)}(c^{-1}t).\]
If $j\in \{2,...,N\}$, we know by choice of $\eta_1,...,\eta_N$ that $\rank P^{t-m,t+m}\ge j-1$ if and only if $t-m\ge 0$ and $t+m<\eta_j$.  This is true if and only if $m\le t$ and $m<\eta_j-t$.  Thus, by definition of the landscape functions, for $j\in \{2,...,N\}$,
\[h_n^{(j)}(t) = \max\{t,\eta_j-t\} = \tau_{(0,\eta_j)}.\]
By convention, $h_n^{(1)} = \tau_{(0,\eta_1)}$.  

Putting everything together, we have established that
\begin{equation}\label{eq : hdef}
    \begin{split}
        h_n^{(1)} = \tau_{(0,\eta_1)},\:h^{(2)}_n = \tau_{(0,\eta_2)},...,h^{(N-1)}_{n} = \tau_{(0,\eta_{N-1})}, h^{(N)}_{n} &= \tau_{(0,\eta_N)}\\
        h^{(kN+1)}_{n}= h^{(kN+2)}_n = ...= h^{(kN+N-1)}_n=h^{(kN+N)}_n &= cg^{(k+1)}_n(c^{-1}x),\:\:\text{for}\:\:k\in\N.
    \end{split}
\end{equation}
This means we can compute
\[\|\mathcal{L}\mathbf{g}_n -\mathbf{h}_n\|_{L^\infty(\N\times\R)} = \max_{1\le j\le N}\|\tau_{(0,\delta_j)} -\tau_{(0,\eta_j)}\|_\infty=\max_{1\le j\le N}|\delta_j - \eta_j| \le \frac{2}{n} .  \]
Putting our three bounds together with Equation~\eqref{eq : Lf_minus_f}, we have
\[\|\mathcal{L}\mathbf{f} - \mathbf{f}\|_{L^\infty(\N\times\R)} \le \frac{c}{n}+\frac{2}{n}+\frac{1}{n}.\]
Taking the limit as $n\to\infty$, we conclude that $\mathcal{L}\mathbf{f} = \mathbf{f}$.

\end{proof}

Now that we have identified a contraction on $L^\infty(\N\times\R)$ whose fixed point is $\mathbf{f}$, the landscape of $A$, we can compute $\mathbf{f}$ itself by finding $\lim_{n\to\infty}L^n\mathbf{0}$, were $\mathbf{0}$ denotes the zero sequence in $L^\infty(\N\times\R)$.  When we do this, we find the formula for $\mathbf{f}= \{f^{(j)}\}_{n=1}^\infty$ is given by 
\begin{equation}\label{eq : PL_wellseperated}
    f^{(j)} = \begin{cases}
    \tau_{(0,\delta_j)} &\text{if}\:\: 1\le n\le N\\
    \tau_{(0,c^k\delta_l)} &\text{if}\:\: (l-1)N^{k}<j\le l N^{k},\:\:k,l\in\N,\:2\le l\le N\\
    \end{cases}.
\end{equation}
Looking back at Equation~\eqref{eq : Cantor_PL}, since we had $\delta_1 = 1$, and $\delta_2=\frac{1}{3}$, we see that this is consistent with what we found for $\mathcal{C}$.  For good measure, we can check our work.  Choose $j\in\N$.  Let $\hat f^{(j)}$ denote the $j$th term in $\mathcal{L}\mathbf{f}$.  Looking at Equation~\eqref{eq : well-sep_op_def} and Equation~\eqref{eq : PL_wellseperated}, it is clear that for $j\le N$, $\hat{f}^{(j)} = f^{(j)}$.  Assuming that $j>N$, define
\[k := \min\{k_0\in\N|j\le N^{k_0}\}\ge 2.\]
Using division, we find $j= N^{k-1}q+r$, for $0\le r<N^{k-1}$.  By assumption, we have $1\le q\le N$.  Indeed, if $q>N$, then $j > N^k$, a contradiction.  If $q = 0$, then $j =r < N^{k-1}$, another contradiction.

We consider 2 cases.  First, if $r=0$, then $q\ge 2$, because otherwise if $q=1$, then $j = N^{k-1}$, but $k-1<k$.  Hence, 
\[j = qN^{k-1}= N(qN^{k-2}-1)+N.\]
According to Equation~\eqref{eq : PL_wellseperated}, this means $f^{(j)} = \tau_{(0,c^{k-1}\delta_q)}$.  On the other hand, for all $t\in \R$, according to Equation~\eqref{eq : well-sep_op_def},
\[\hat{f}^{(j)}(t)= cf^{(qN^{k-2})}(c^{-1}t) = c\tau_{(0,c^{k-2}\delta_q)}(c^{-1}t) = \tau_{(0,c^{k-1}\delta_q)}(t).\]
Thus $f^{(j)} = \hat{f}^{(j)}$.

In the case of $r>0$, we apply division to $r$, so that $r = Na+b$, $a,b\in \N$, with $0\le a<N^{k-2}$, and $0\le b<N$.  This means that
\[j = qN^{k-1}+Na+b = N(qN^{k-2}+a-1)+N+b.\]
Since $r>0$, if $q = N$, then $j= N^k+r>N^k$, a contradiction.  Thus $1\le q\le N-1$.  Since $qN^{k-1}<j<(q+1)N^{k-1}$, by Equation~\eqref{eq : PL_wellseperated}, we have
\[f^{(j)} = \tau_{(0,c^{k-1}\delta_{q+1})}.\]
To make things easier, we consider 2 subcases.  If $b= 0$, then $a>0$, and $j = N(qN^{k-2}+a-1)+N$.  By Equation~\eqref{eq : well-sep_op_def} we have for all $t\in \R$
\[\hat{f}^{(j)}(t) = cf^{(qN^{k-2}+a)}(c^{-1}t)=c\tau_{(0,c^{k-2}\delta_{q+1})}(c^{-1}t)=\tau_{(0,c^{k-1}\delta_{q+1})}(t).\]
Hence $f^{(j)}= \hat f^{(j)}$.  If $b>0$, then we have $j = N(qN^{k-2}+a)+b$.  Since $qN^{k-2}<qN^{k-2}+a+1\le (q+1)N^{k-2}$, we have for all $t\in\R$
\[\hat{f}^{(j)}(t) = cf^{(qN^{k-2}+a+1)}(c^{-1}t)=c\tau_{(0,c^{k-2}\delta_{q+1})}(c^{-1}t) = \tau_{(0,c^{k-1}\delta_{q+1})}(t).\]
Thus $\hat{f}^{(j)} = f^{(j)}$ in this final case.  

Knowing that Equation~\eqref{eq : PL_wellseperated} is the correct formula for the persistence landscape for $H_0\CechMod(A)$, when $\Psi$ has well-separated images gives us a head start for finding the persistence landscape related to many IFS.  In practice, it can be difficult to compute the resolutions $\delta_1,...,\delta_N$ using only the functions in $\Psi$.  As we saw with $\mathcal{C}$, this can be straightforward in one-dimension.  We will apply Equation~\eqref{eq : PL_wellseperated} to more precisely describe the persistence landscape of $H_0\CechMod(A)$ when $\Psi$ is an IFS on $\R$ before looking at some more interesting examples.

Theorem \ref{thm : Op_WSI_IFS} applies to all IFS which satisfy the well-separated condition, but it is not as strong of a result as Theorem \ref{thm : cantorset_op}.  For that, we need an additional assumption on the IFS.

\begin{theorem}\label{thm : IFS_Op_Invariance}
Let $\Psi$ be a IFS on $\R^d$ consisting of $N$ similitudes each with scaling constant $c\in(0,1)$.  Let $A$ be the invariant set of $\Psi$ and let $\mathbf{f}_n\in L^\infty(\N\times\R)$ denote the persistence landscape resulting from $H_0\CechMod(S_n)$, where $S_{0} = E_{A}$.  Suppose that $\Psi$ has well-separated images and also satisfies the property that for all $j,k = 1,\dots,N$ and for all $n \in \mathbb{N}$,
\begin{equation}\label{eq : const_sep}
d(\psi_{j}(S_{n+1}), \psi_{k}(S_{n+1}) ) = d(\psi_{j}(S_{n}), \psi_{k}(S_{n}) ). 
\end{equation}
Let $\mathcal{L}:L^\infty(\N\times\R)\to L^\infty(\N\times\R)$ as defined in Equation~\eqref{eq : well-sep_op_def}. Then for all $n \in \mathbb{N}$,  $\mathcal{L}\mathbf{f}_{n} = \mathbf{f}_{n+1}$.
\end{theorem}

In other words, the commutative diagram in Equation \eqref{eq : com_dgm0} holds.

\begin{proof}
Let $\mathcal{L}\mathbf{f}_n = \left\{\hat{f}_n^{(j)}\right\}_{j=1}^\infty$. We must show that for all $j\in \N$
\begin{equation}\label{eq : elementwise}
    \hat{f}_n^{(j)} = f_{n+1}^{(j)}.
\end{equation}
For the case $j = 1$, this is true because $\diam(S_{n+1}) = \diam(A)$ implies
\[f_{n+1}^{(1)} = \tau_{(0,\delta_1)}.\]
For the other $j\in\{2,...,N\}$, Equation~\eqref{eq : elementwise} is equivalent to showing that for $\varepsilon \ge \delta_N$
\begin{equation}\label{eq : eps_components}
C(S_{n+1}, \varepsilon) = C(A,\varepsilon).
\end{equation}
Indeed, if $\varepsilon \ge \delta_N$, then by the well-separated assumption, each image $\psi_j(A)$ is contained in an  $\varepsilon$-component of $A$.  To prove Equation~\eqref{eq : eps_components}, define an equivalence relation, $\sim_A$, on the images $\{\psi_j(A)\}_{j=1}^N$ by saying
$\psi_j(A) \sim_{A} \psi_k(A)$ if and only if $\psi_j(A)$ and $\psi_k(A)$ belong to the same $\varepsilon$-component.  The number of distinct classes in $\{ \psi_{j}(A) \}/ \sim_{A}$ is equal to $C(A,\varepsilon)$ since $A = \cup_{j=1}^N\psi_j(A)$.  Similarly, we can define an equivalence relation, $\sim_{n+1}$, on $\{\psi_j(S_n)\}_{j=1}^N$ based on the $\varepsilon$-components of $S_{n+1}$ so that the number of distinct classes of $\sim_{n+1}$ equals $C(S_{n+1},\varepsilon)$.  We claim that for $j,k\in\{1,...,N\}$, 
\[\psi_{j}(A)\sim_A\psi_k(A)\:\:\text{if and only if}\:\:\psi_j(S_{n})\sim_{n+1}\psi_k(S_{n}).\]
Equation \eqref{eq : const_sep} implies that for all $n\in\N$
\begin{equation} \label{eq : const_sep2}
d(\psi_j(S_n),\psi_k(S_n)) = d(\psi_j(A),\psi_{k}(A)).
\end{equation}
In addition, $\psi_{j}(A)\sim_A\psi_k(A)$ implies that there is a finite sequence of images $\psi_{j}(A) = \psi_{j_0}(A), \psi_{j_1}(A), \dots, \psi_{j_m}(A) = \psi_{k}(A)$ with 
\[d(\psi_{j_l}(A),\psi_{j_{l+1}}(A)) \leq \varepsilon.\]
It now follows that $\psi_{j}(A)\sim_A\psi_k(A)$ implies  $\psi_j(S_{n})\sim_{n+1}\psi_k(S_{n})$.  The reverse implication follows similarly.  We have established that Equation~\eqref{eq : elementwise} holds for $j\le N$.

Now let $j >N$.  We write $j = kN + l$ for $k\in \N$ and $l\in \{1,...,N\}$.  For $0\le s\le t$, let $M^{s,t}$ and $P^{s,t}$ denote the homomorphisms induced by the inclusion mappings $\Cech(S_{n},s)\to\Cech(S_{n},t)$ and $\Cech(S_{n+1},s)\to \Cech(S_{n+1},t)$ respectively.  By Lemma~\ref{lemma : dir_sum_hom} along with the fact that both induced homomorphisms are surjective, for $t\le \delta_N$, 
\[\rank P^{s,t} = N\rank M^{c^{-1}s,c^{-1}t}.\]
By the same reasoning as in the proof of Theorem~\ref{thm : Op_WSI_IFS}
\begin{equation}
\begin{split}
    f_{n+1}^{(j)}(x) &= \sup\{m\ge 0 |\rank P^{x-m,x+m}\ge kN+l\}\\
    &= \sup\{m\ge 0 |\rank M^{c^{-1}x-m,c^{-1}x+m}\ge k+lN^{-1}\}\\
    &= c\sup\{m\ge 0 |\rank M^{c^{-1}x-m,c^{-1}x+m}\ge k+1\}
\end{split}
\end{equation}
and by the definition of $\mathcal{L}$, we have
\begin{equation*}
    \hat{f}_n^{(j)}(x)= cf_n^{(k+1)}(c^{-1}x)= c\sup\{m\ge 0 |\rank M^{c^{-1}x-m,c^{-1}x+m}\ge k+1\} = f_{n+1}^{(j)}(x). \qedhere
\end{equation*}
\end{proof}

\subsection{Special Case: Dimension One}

Here we assume $\Psi = \{\psi_j\}_{j=1}^N$ is an IFS where each $\psi_j:\R\to\R$ is defined by $\psi_j(x) = c(x-b_j)$, for some $c\in(0,1)$, $b_1,...,b_N\in\R$.  We assume without loss of generality that $b_j\le b_{j+1}$.  For $j\in \{1,...,N-1\}$, we also define $a_j := b_{j+1}-b_j$.  We know that for each $j$, the fixed point of $\psi_j$ is \[x_j = \frac{cb_j}{1-c}.\]
From Lemma~\ref{lemma : extreme_fixed} we know that the extreme points of the $\conv A$ will be $x_1$ and $x_N$.  This means, 
\begin{equation*}
    \delta_1 = \diam A = \frac{c}{1-c}(b_N-b_1).
\end{equation*}
Also, for $j\not= k$, we know that $\psi_j(A)$ is just a translation of $\psi_k(A)$ by $c(b_j-b_k)$.  Hence,
\[d(\psi_j(A),\psi_k(A)) = c(b_j-b_k) - c\delta_1 = c(b_j-b_k)-\frac{c^2}{1-c}(b_N-b_1).\]
For each $j$, the closest other image to $\psi_j(A)$ is either $\psi_{j+1}(A)$ or $\psi_{j-1}(A)$.  Thus we know that $\Psi$ has well separated images if 
\begin{equation}
    \begin{split}
        \frac{c^2}{1-c}(b_N-b_1) &= \diam(\psi_j(A)) \le \min_{1\le j\le N-1}d(\psi_j(A),\psi_{j+1}(A))\\
        &= \min_{1\le j\le N-1}ca_j-\frac{c^2}{1-c}(b_N-b_1).
    \end{split}
\end{equation}
This is equivalent to 
\begin{equation}
    2\delta_1 = \frac{2c}{1-c}(b_N-b_1)\le \min_{1\le j\le N-1}a_j.
\end{equation}
It is also straightforward to compute $\delta_2,...,\delta_N$ in the one-dimensional case.

\begin{proposition}\label{prop : 1d_deltas}
Let $\Psi = \{\psi_j\}_{j=1}^N$ be an IFS consisting of similitudes on $\R$ with well-separated images and invariant set $A$.  Then $\delta_2,...,\delta_N$ as defined in Equation~\eqref{eq : resolutions_def}, are also given by
\[\delta_{k+1} = \kmax\{c(a_j-\delta_1)|\:1\le j\le N-1\}.\]
\end{proposition}

\begin{proof}
For convenience, let $\rho_k = \kmax\{c(a_j-\delta_1)|\:1\le j\le N-1\}$ for $k\in\{2,...,N\}$. First we claim that for $\varepsilon>0$ and $N-1\ge k\ge 1$, $\varepsilon\ge \rho_k$ if and only if $C(A,\varepsilon)\le k$.  Indeed, first if we assume $\varepsilon\ge \rho_k$, then 
\[|\{c(a_j-\delta_1)>\varepsilon|\:1\le j\le N-1\}| < k.\]
Our assumption also implies that $\varepsilon \ge c\delta_1$, since otherwise, because $\Psi$ has well-separated images,
\[\begin{split}
    \varepsilon&<c\delta_1 = \max_{1\le j\le N}\diam\psi_j(A)\le \min_{1\le j\not=k\le N}d(\psi_j(A),\psi_k(A))\\
    &= \text{Nmax}\{c(a_j-\delta_1)|\:1\le j\le N-1\}\le \rho_k,
    \end{split}\]
a contradiction.  Since $\varepsilon \ge \diam\psi_j(A)$ for all $j\in\{1,...,N\}$, each image $\psi_j(A)$ is contained in exactly 1 $\varepsilon$-component of $A$.  We can count the $\varepsilon$-components in the following way.  Start the count at $q = 1$.  For each $j \in\{1,...,N-1\}$, we check the distance between $\psi_j(A)$ and $\psi_{j+1}(A)$.  If $d(\psi_j(A),\psi_{j+1}(A))\le\varepsilon$, then $\psi_{j+1}(A)$ belongs to the same $\varepsilon$-component, so the count $q$ remains at the current value.  If $d(\psi_j(A),\psi_{j+1}(A))>\varepsilon$, then $\psi_{j+1}(A)$ belongs to a different $\varepsilon$-component from that containing $\psi_l(A)$ for $1\le l\le j$.  In this case we update $q$ to equal $q+1$.  Since $d(\psi_j(A),\psi_{j+1}(A)) = c(a_j-\delta_1)$, we will update $q$ at most $k-1$ times.  Thus $C(A,\varepsilon)\le k$.

Conversely, if $\varepsilon <\rho_k$, then 
\begin{equation}\label{eq : big_gaps}
|\{c(a_j-\delta_1)>\varepsilon|1\le j\le N-1\}|\ge k.
\end{equation}
We consider two cases, first if $\varepsilon < c\delta_1$, then again since $\Psi$ has well-separated images we have
\[\varepsilon<\min_{1\le j\not=k\le N}d(\psi_{j}(A),\psi_k(A)).\]
This implies that each image $\psi_j(A)$ contains at least 1 $\varepsilon$-component of $A$.  Hence $C(A,\varepsilon)\ge N> k$ as desired.  In the second case, we may have $\varepsilon\in[c\delta_1,\rho_k)$, then since 
\[\varepsilon \ge \max_{1\le j\le N}\diam\psi_j(A),\]
we may repeat the counting of connected components as described in the previous paragraph.  We start with the count at $q=1$.  Equation~\eqref{eq : big_gaps} implies that we must update $q$ at least $k$ times.  Therefore $C(A,\varepsilon)\ge k+1>k$ as desired.  This proves the claim.

It follows from the claim that $\rho_k\in\{\varepsilon>0|C(A,\varepsilon)\le k\}$.  Thus $\delta_{k+1} \le \rho_k$.  Conversely, since $\rho_k$ is a lower bound, we have that 
\[\rho_k \le \inf\{\varepsilon>0|C(A,\varepsilon)\le k\} = \delta_{k+1}.\qedhere\]
\end{proof}

Now we are ready to state the main consequence of Theorem \ref{thm : Op_WSI_IFS} for IFS on $\R$.

\begin{corollary}\label{cor : 1d_IFS_PL}
Let $\Psi = \{\psi_j\}_{j=1}^N$ be an IFS of similitudes on $\R$ of the form $\psi_j(x) = c(x+b_j)$, for $c\in(0,1)$ and $b_j\in\R$.  Assume $b_j\le b_{j+1}$ and let $a_j = b_{j+1}-b_j$.  $\Psi$ has well-separated images if and only if 
\begin{equation}\label{eq : 1d_well-separated}
    \frac{2c}{1-c}(b_N-b_1)\le \min_{1\le j\le N-1}a_j.
\end{equation}
Moreover, if Equation~\eqref{eq : 1d_well-separated} holds and $A$ is the invariant set of $\Psi$, then $\mathbf{f}= \{f^{(j)}\}_{j=1}^\infty\in L^\infty(\N\times\R)$, the persistence landscape of $H_0\CechMod(A)$, is given by
\begin{equation}\label{eq : 1d_IFS_PLfunction}
    f^{(j)} = \begin{cases}
    \tau_{(0,\delta_j)} &\text{if}\:\: 1\le j\le N\\
    \tau_{(0,c^k\delta_l)} &\text{if}\:\: (l-1)N^{k}<j\le l N^{k},\:\:k,l\in\N,\:2\le l\le N\\
    \end{cases},
\end{equation}
where $\delta_1 = \frac{c}{1-c}(b_N-b_1)$, and for $k\in\{1,...,N-1\}$,
\[\delta_{k+1} = \kmax\{c(a_l-\delta_1)|\:1\le l\le N-1\}.\]
\end{corollary}
\begin{proof}
As reasoned above, $\Psi$ having well-separated images is equivalent to Equation~\eqref{eq : 1d_well-separated}.  It is also explained above that $\delta_1=\diam(A) = \frac{c}{1-c}(b_N-b_1)$, and the formula for the other $\delta_k$'s is a consequence of Proposition~\ref{prop : 1d_deltas}.  The formula in Equation~\eqref{eq : 1d_IFS_PLfunction} follows from Theorem~\ref{thm : Op_WSI_IFS}.
\end{proof}

\section{Examples}
We are now ready to present a series of examples of iterated function systems and the corresponding persistence landscapes resulting from the invariant set.  Our goal is to illustrate the relationship between the persistence landscape and the IFS.  Some of our examples will have well-separated images, meaning that we can readily apply the results above to compute the persistence landscape of $H_0\CechMod(A)$.  Other examples will require some additional work, but the reasoning should be similar to that used to prove Theorem~\ref{thm : Op_WSI_IFS}.  We are also able to check our work by approximating $A$ and computing the persistence landscape using the Scikit-TDA library in Python.


\subsection{Right 1/3 Cantor Set} \label{ssec:right-Cantor}
Consider the IFS $\Psi = \{\psi_1,\psi_2\}$ where
\[\psi_1(x) = \frac{1}{3}x,\:\:\psi_2(x) = \frac{1}{3}x+\frac{1}{3}.\]
In this case, we have $c = \frac{1}{3}$, $b_1 = 0$, and $b_2=1$.  This means $\Psi$ has well-separated images since
\[\frac{2c}{1-c}(b_2-b_1) = 1 = \min_{1\le j\le N-1}a_j.\]
Clearly, $\delta_1 = \frac{1}{2}$, and 
\[\delta_2 = \frac{1}{3}(1-\delta_1) = \frac{1}{6}\]
By Corollary~\ref{cor : 1d_IFS_PL}, this means the persistence landscape of $H_0\Cech(A)$ is $\mathbf{f} = \{f^{(j)}\}_{j=1}^\infty$, where
\[ f^{(j)} = \begin{cases}
    \tau_{(0,\tfrac{1}{2})} &\text{if}\:\: j=1\\
    \tau_{(0,\tfrac{1}{6})} &\text{if}\:\: j=2\\
    \tau_{(0,3^{-k}\tfrac{1}{6})} &\text{if}\:\: 2^{k}<j\le 2^{k+1},\:\:k\in\N,\:\\
    \end{cases}
\]

\subsection{1/5 Cantor Set} \label{ssec:fifth-Cantor}
Consider the IFS $\Psi = \{\psi_1,\psi_2,\psi_3\}$ where
\[\psi_1(x) = \frac{1}{5}x,\:\:\psi_2(x) = \frac{1}{5}x+\frac{2}{5},\:\:\psi_3(x) = \frac{1}{5}x+\frac{4}{5}.\]
In this case, we have $c = \frac{1}{5}$, $b_1 = 0$, $b_2 = 2$, and $b_3 = 4$.  Again, we know that $\Psi$ has well-separated images since 
\[\frac{2c}{1-c}(b_3-b_1) = 2 = \min_{1\le j\le N-1}a_j.\]\
Clearly, $\delta_1 = 1$ and 
\[\delta_2 = \delta_3 = \frac{1}{5}(2-\delta_1) = \frac{1}{5}.\]
By Corollary~\ref{cor : 1d_IFS_PL}, this means that the persistence landscape of $H_0\Cech(A)$ is $\mathbf{f} = \{f^{(j)}\}_{j=1}^\infty$, where
\[f^{(j)} = \begin{cases}
    \tau_{(0,1)} &\text{if}\:\: j=1\\
    \tau_{(0,\tfrac{1}{5})}  &\text{if}\:\:j=2,3\\
    \tau_{(0,5^{-k-1})} &\text{if}\:\: 3^{k}<j\le  3^{k+1},\:\:k\in\N
    \end{cases}\]
    
\subsection{1/6 Cantor Set} \label{ssec:sixth-Cantor}
Consider the IFS $\Psi = \{\psi_1,\psi_2,\psi_3\}$ where
\[\psi_1(x) = \frac{1}{6}x,\:\:\psi_2(x) = \frac{1}{6}x+\frac{2}{6},\:\:\psi_3(x) = \frac{1}{6}x+\frac{5}{6}.\]
In this case, we have $c = \frac{1}{6}$, $b_1 = 0$, $b_2 = 2$, and $b_3 = 5$.  We know that $\Psi$ has well-separated images since 
\[\frac{2c}{1-c}(b_3-b_1) = 2 = \min_{1\le j\le N-1}a_j.\]\
We compute $\delta_1 = 1$, 
\[\delta_2 = \frac{1}{6}(3-\delta_1) = \frac{1}{3},\:\:\text{and}\:\:\delta_3= \frac{1}{6}(2-\delta_1) = \frac{1}{6}.\]
By Corollary~\ref{cor : 1d_IFS_PL}, this means that the persistence landscape of $H_0\Cech(A)$ is $\mathbf{f} = \{f^{(j)}\}_{n=1}^\infty$, where
\[f^{(j)} = \begin{cases}
    \tau_{(0,1)} &\text{if}\:\: j=1\\
    \tau_{(0,\tfrac{1}{3})} &\text{if}\:\: j=2\\
    \tau_{(0,\tfrac{1}{6})} &\text{if}\:\: j=3\\
    \tau_{(0,6^{-k}\tfrac{1}{3})} &\text{if}\:\: 3^{k}<j\le 2\cdot 3^{k},\:\:k\in\N,\\
    \tau_{(0,6^{-k-1})} &\text{if}\:\: 2\cdot 3^{k}<j\le 3^{k+1},\:\:k\in\N\\
    \end{cases}\]

\subsection{Modified 1/5 Cantor Set} \label{ssec : mod1/5}

Let 
\[\Psi = \{\psi_1,\psi_2,\psi_3\},\:\:\psi_1(x) = \frac{1}{5}x,\:\psi_2(x) = \frac{1}{5}(x+1),\:\:\psi_3(x) = \frac{1}{5}(x+4).\]
In this case, $b_1 = 0$, $b_2 = 1$, and $b_3 = 4$.  This means $\delta_1 = 1$.  $\Psi$ does not have well separated images since $d(\psi_1(A),\psi_2(A)) = 0$ because $\frac{1}{5}\in\psi_1(A)\cap\psi_2(A)$.  Although Theorem~\ref{thm : Op_WSI_IFS} does not apply, we still claim that the map $\mathcal{L}:L^\infty(\N\times\R)\to L^\infty(\N\times\R)$ defined by $\mathcal{L}\mathbf{g} = \mathbf{h}$, where
\[\begin{split}
h^{(1)} = \tau_{(0,1)},\:\: h^{(2)} &= \tau_{(0,2/5)}\\
h^{(3k)}(x) = h^{(3k+1)}(x) = h^{(3k+2)}(x) &= \frac{1}{5}g^{(k+1)}(5x)\:\:\text{for}\:\:k\in\mathbb{N}.
\end{split}\]  
satisfies $\mathcal{L}\mathbf{f} = \mathbf{f}$ where $\mathbf{f}$ is the persistence landscape resulting from $H_0\Cech(A)$.  To see why, we construct an increasing sequence of sets as follows
\[S_1=\left\{0,\frac{1}{5}, \frac{2}{5},\frac{4}{5},1\right\}, S_{n+1} = \Psi(S_n)\:\:\text{for}\:\:n\in\N.\]
It follows from Lemma~\ref{lemma : extreme_fixed} that $S_n\subset S_{n+1}\subset A$ for all $n\in\N$.  The first containment follows from the fact that $S_1 = \Psi(E_A)$.  Let $\mathbf{f}_n\in L^\infty(\N\times\R)$ be the persistence landscape of $H_0\Cech(S_n)$. 

We claim that for all $n\in\N$, $\mathcal{L}\mathbf{f}_n = \mathbf{f}_{n+1}$.  Indeed, fix $n\in\N$.  First observe that 
\[\diam(S_n) = \diam(A) = 1.\]
Thus $f_{n+1}^{(1)} = \tau_{(0,1)}$.  As before for $s\le t$, let $M^{s,t}:H_0\Cech(S_n,s)\to H_0\Cech(S_n,t)$ and $P^{s,t}:H_0\Cech(S_{n+1},s)\to H_0\Cech(S_{n+1},t)$ denote the homomorphisms induced by the obvious inclusion mappings. 
Since $\Cech(S_{n+1},\varepsilon)$ is path connected for $\varepsilon\ge \frac{2}{5}$, this implies that $\rank H_0\Cech(S_{n+1},\varepsilon) = 1$.  On the other hand, for $\varepsilon <\frac{2}{5}$, we know that $\rank H_0\Cech(S_{n+1},\varepsilon)\ge 2$ since $\Cech(S_{n+1},\varepsilon)$ is not path connected.  In the final case, we know that $\rank H_0\Cech(S_{n+1},\varepsilon) = 0$ for $\varepsilon <0$.  Thus it follows by definition and the same reasoning used in the proof of Theorem~\ref{thm : Op_WSI_IFS} that
\begin{equation}
    f_{n+1}^{(2)}(t) = \sup\{m\ge 0|\rank P^{t-m,t+m}\ge 2\} = \max\{t,\tfrac{2}{5}-t\} = \tau_{(0,\frac{2}{5})}.
\end{equation}
To compute $f_{n+1}^{(j)}$ for $j\ge 3$, the key observation is that for $0\le s\le t<\frac{2}{5}$, 
\begin{equation}\label{eq : rankPst}
    \rank P^{s,t} = 3\rank M^{5s,5t} -1.
\end{equation}
Also, it is clear that for $s <0$ or $t\ge \frac{2}{5}$, $\rank P^{s,t} \le 1$.  To justify Equation~\eqref{eq : rankPst}, we look at $\Cech(S_{n+1},\varepsilon)$ for $\varepsilon\in (0,\frac{2}{5})$.  For $m\in\{1,2,3\}$ let $K_m$ denote  $\Cech(\psi_m(S_{n}),\varepsilon)\subset\Cech(S_{n+1},\varepsilon)$.  A simple illustration of this is given in Figure~\ref{fig : mod1/5}.  Since $\varepsilon <\frac{2}{5}$, $K_3$ is not path connected to $K_1\cup K_2$.  By Lemma~\ref{lemma : dir_sum_hom} and Corollary~\ref{cor : sim_hom_isomorphism}, we see that
\begin{equation}\label{eq : HcechSj+1}
    H_0\Cech(S_{n+1},\varepsilon)\cong H_0(K_1\cup K_2)\oplus H_0(K_3)\cong H_0(K_1\cup K_2)\oplus H_0\Cech(S_n,5\varepsilon).
\end{equation}
From the Mayer-Vietoris sequence, we have the following exact sequence
\[H_1(K_1\cup K_2)\longrightarrow H_0(K_1\cap K_2)\longrightarrow H_0(K_1)\oplus H_0(K_2)\longrightarrow H_0(K_1\cup K_2)\longrightarrow 0.\]
Since $S_{n+1}\subset\R$, and $K_1\cap K_2 = \{\frac{1}{5}\}$, this sequence is equivalent to
\begin{equation}
    0\longrightarrow \Z_2\overset{\varphi_1}{\longrightarrow} H_0(K_1)\oplus H_0(K_2)\overset{\varphi_2}{\longrightarrow} H_0(K_1\cup K_2)\longrightarrow 0.
\end{equation}
Exactness implies that $\varphi_2$ is surjective and $\ker\varphi_2\cong \Z_2$.    Thus 
\[\dim H_0(K_1\cup K_2) = \dim H_0(K_1)+\dim H_0(K_2) -1.\]
Since $H_0(K_m)\cong H_0\Cech(S_n,5\varepsilon)$ for $m\in\{1,2,3\}$, Equation~\eqref{eq : HcechSj+1} implies that $$\dim[H_0\Cech(S_{n+1},\varepsilon)] = 3\dim[H_0\Cech(S_n,5\varepsilon)]-1.$$  Since $M^{s,t}$ and $P^{s,t}$ are always surjective when $s>0$, Equation~\eqref{eq : rankPst} follows.

Now we may argue as we did for Theorem~\ref{thm : Op_WSI_IFS}.  For $j= 3k+m$ with $k\in\N$, $m\in\{0,1,2\}$, by definition
\begin{equation}
    \begin{split}
    f_{n+1}^{(j)}(t) &= \sup\{m\ge 0 | \rank P^{t-m,t+m}\ge j\}\\
    &= \sup\{m\ge 0|3\rank M^{5(t-m),5(t+m)}-1\ge j\}\\
    &= \sup\{m\ge 0|\rank M^{5(t-m),5(t+m)}\ge k+1\}
    \end{split}
\end{equation}
On the other hand, for $k\in\N$, we have
\[\begin{split}
    \frac{1}{5}f_n^{(k+1)}(5t)& = \frac{1}{5}\sup\{m\ge 0|\rank M^{5t-m,5t+m}\ge k+1\}\\
    &= \sup\{\tfrac{m}{5}\ge 0|\rank M^{5t-m,5t+m}\ge k+1\}\\
    &= \sup\{m\ge 0|\rank M^{5(t-m),5(t+m)}\ge k+1\}\\
    & = f_{n+1}^{(j)}(t).
\end{split}\]
This proves the claim that $\mathcal{L}\mathbf{f}_n = \mathbf{f}_{n+1}$.  

\begin{figure}
    \centering
    \includegraphics[scale = 0.65]{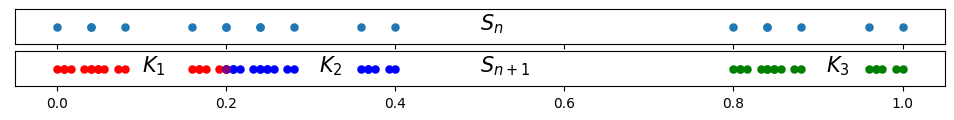}
    \caption{The three sub-complexes in the modified 1/5 Cantor set  for $n = 3$.}
    \label{fig : mod1/5}
\end{figure}

From Theorem~\ref{thm : Hutch1981}, we know that $\lim_{n\to\infty}d_H(S_n,A) = 0$.  Thus $\lim_{n\to\infty}\mathbf{f}_n = \mathbf{f}$.  Since $\mathcal{L}$ is Lipschitz on $L^\infty(\N\times\R)$, it is continuous.  Thus
\[\mathcal{L}\mathbf{f} = \lim_{n\to\infty}\mathcal{L}\mathbf{f}_n = \lim_{n\to\infty}\mathbf{f}_{n+1} = \mathbf{f}.\]
as claimed.

To obtain the formula for $\mathbf{f} = \{f^{(j)}\}_{j=1}^\infty$, we compute $\lim_{n\to\infty}\mathcal{L}^n\mathbf{0}$ in $L^\infty(\N\times\R)$.  We find that

\[f^{(j)} = \begin{cases}
\tau_{(0,1)}   & \text{if}\:\: j= 1\\
\tau_{(0,\tfrac{2}{5})} & \text{if}\:\: j=2\\
\tau_{(0,\tfrac{2}{5^{k+2}})} & \text{if}\:\: 2+\tfrac{3}{2}(3^k-1)<j\le 2+\tfrac{3}{2}(3^{k+1}-1),\:\text{where}\:k+1\in\N 
\end{cases}.\]

We illustrate this landscape in Figure \ref{fig : cantor-5} (produced by pyscapes~\cite{angeloro_pyscapes_2020}).

\begin{figure}[h!]
    \includegraphics[width=12cm]{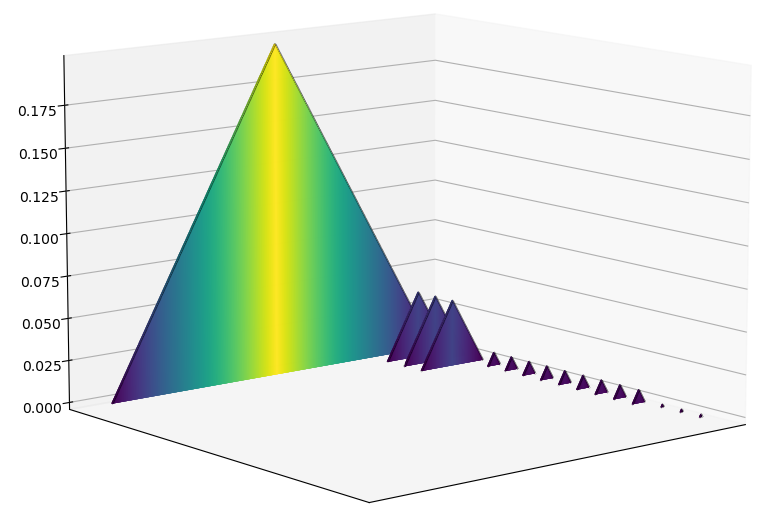}
    \caption{Graph of the functions $f_{2}, \dots, f_{16}$ from the persistence landscape of the modified $1/5$ Cantor set $\mathcal{C}$. }
    \label{fig : cantor-5}
\end{figure}

\subsection{Cantor Triangle} \label{ssec:Cantor-triangle}
Let us consider a 2-dimensional example.  Consider the IFS on $[0,1]^2$, $\Psi = \{\psi_1, \psi_2,\psi_3\}$, where
\[\psi_1(x,y) = \frac{1}{3}I_2(x,y)^T,\:\:\psi_{2}(x,y) = \frac{1}{3}I_2[(x,y)+(0,2)]^T,\:\:\psi_{3}(x,y) = \frac{1}{3}I_2[(x,y)+(2,0)]^T,\]
with $I_2$ denoting the $2\times 2$ identity matrix.  The set of extreme points is $S_0 = \{(0,0), (0,1),(1,0)\}$.  However $\Psi$ does not have well-separated images since
\[\min_{1\le j\not=k\le N}d(\psi_j(A),\psi_k(A))= \frac{1}{3}<\frac{\sqrt{2}}{3}= \max_{1\le j\le N}\diam \psi_j(A).\]
Despite this, the formula in Equation~\eqref{eq : PL_wellseperated} still applies. To see why, define $\mathcal{L}:L^\infty(\N\times\R)\to L^\infty(\N\times\R)$ by $\mathcal{L}\mathbf{g} = \mathbf{h}$, where
\[\begin{split}
h^{(1)} &= \tau_{(0,\sqrt{2})},\:\: h^{(2)} = h^{(3)} = \tau_{(0,1/3)} ,\\
h^{(3k)}(x) &= h^{(3k+1)}(x) = h^{(3k+2)}(x) = \frac{1}{3}g^{(k+1)}(3x)\:\:\text{for}\:\:k\in\mathbb{N}.
\end{split}.\]
Note that as in Equation~\eqref{eq : PL_wellseperated}, we have $\delta_1 = \sqrt{2}$ and $\delta_2= \delta_3 = \frac{1}{3}$.  

Define a sequence of sets $\{S_n\}_{n=1}^\infty$ by $S_1 =\Psi(S_0)$, and $S_{n+1} = \Psi(S_n)$ for $n\in\N$.  Note that $S_n\subset S_{n+1}\subset A$ for all $n\in\N$ and $\lim_{n\to\infty}d_H(S_n,A) = 0$.  Let $\mathbf{f}_n\in L^\infty(\N\times\R)$ be the persistence landscape of $H_0\Cech(S_n)$.  As before, $\lim_{n\to\infty}\mathbf{f}_n = \mathbf{f}$ where $\mathbf{f}$ is the persistence landscape of $H_0\Cech(A)$.  Thus, we claim $\mathcal{L}\mathbf{f}_n = \mathbf{f}_{n+1}$.  

\begin{figure}
    \centering
    \includegraphics{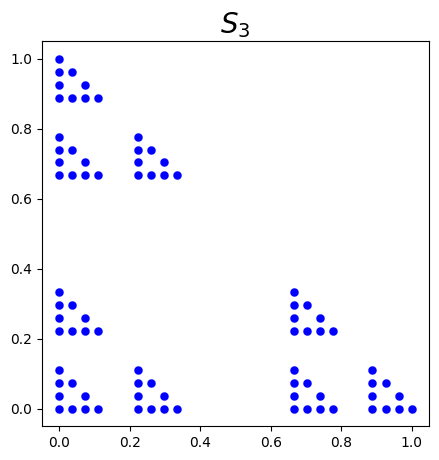}
    \caption{An illustration of $S_3$ for the Cantor triangle.}
    \label{fig : cantor_tri}
\end{figure}

Indeed, choose $n\in\N$.  Since $\diam(A) = \sqrt{2}$, it is clear that $f_{n+1}^1 =\tau_{(0,\sqrt{2})}$.  For $s\le t$, let $M^{s,t}:H_0\Cech(S_n,s)\to H_0\Cech(S_n,t)$ and $P^{s,t}:H_0\Cech(S_{n+1},s)\to H_0\Cech(S_{n+1},t)$ denote the homomorphisms induced by the obvious inclusion mappings.
Since $\Cech(S_{n+1},\varepsilon)$ is path connected for $\varepsilon\ge \frac{1}{3}$, this implies that $\rank H_0\Cech(S_{n+1},\varepsilon) = 1$.  On the other hand, for $\varepsilon <\frac{1}{3}$, we know that $\rank H_0\Cech(S_{n+1},\varepsilon)\ge 3$ since $\Cech(S_{n+1},\varepsilon)$ is not path connected, with $\psi_1(S_n),\psi_2(S_n),\psi_3(S_n)\subset S_{n+1}$ each containing at least one distinct $\varepsilon$-component of $S_{n+1}$.  
Therefore, for $t\in\R$, $j = 2$ or $j = 3$
\[f_{n+1}^{(j)}(t) = \sup\{m\ge 0 \, | \, \rank P^{t-m,t+m}\ge 3\} = \max\{t,\tfrac{1}{3}\} = \tau_{(0,\tfrac{1}{3})}(t).\]

To compute $f_{n+1}^{(j)}$ for $j\ge 4$, the key observation is that for $0\le s\le t<\frac{1}{3}$,
\begin{equation}\label{eq : rankeq_triangle}
    \rank P^{s,t} = 3\rank M^{3s,3t},
\end{equation}
and for $s<0$ or $t\ge \frac{1}{3}$, we get $\rank P^{s,t}\le 1$.  The second part is clear.  To justify Equation~\eqref{eq : rankeq_triangle}, note that for $\varepsilon <\frac{1}{3}$, since
\[\min_{1\le j\not= k\le 3}d(\psi_j(S_n),\psi_k(S_n)) = \frac{1}{3},\]
it follows from Lemma~\ref{lemma : dir_sum_hom} that 
\[H_0\Cech(S_{n+1},\varepsilon)\cong\bigoplus_{j=1}^3 H_0\Cech(S_n,3\varepsilon).\]
Since every nontrivial transformation in $H_0 \CechMod(S_n)$ and $H_0 \CechMod(S_{n+1})$
is a surjection, the previous isomorphism implies 
Equation~\eqref{eq : rankeq_triangle}.
If we assume $j = 3k+l$ for $k\in\N$, and $l\in\{1,2,3\}$, then by definition,
\begin{equation}
    \begin{split}
    f_{n+1}^{(j)}(t) &= \sup\{m\ge 0 | \rank P^{t-m,t+m}\ge j\}\\
    &= \sup\{m\ge 0|3\rank M^{3(t-m),3(t+m)}\ge j\}\\
    &= \sup\{m\ge 0|\rank M^{3(t-m),3(t+m)}\ge k+1\}
    \end{split}
\end{equation}
On the other hand, for $k\in\N$, we have
\[\begin{split}
    \frac{1}{3}f_{n}^{(k+1)}(3t)& = \frac{1}{3}\sup\{m\ge 0|\rank M^{3t-m,3t+m}\ge k+1\}\\
    &= \sup\{\tfrac{m}{3}\ge 0|\rank M^{3t-m,3t+m}\ge k+1\}\\
    &= \sup\{m\ge 0|\rank M^{3(t-m),3(t+m)}\ge k+1\}\\
    & = f_{n+1}^{(j)}(t).
\end{split}\]
This proves the claim that $\mathcal{L}\mathbf{f}_n = \mathbf{f}_{n+1}$, meaning $\mathcal{L}\mathbf{f} = \mathbf{f}$.

When we compute the fixed point of $\mathcal{L}$ by taking the limit $\lim_{n\to\infty}\mathcal{L}^n\mathbf{0}$, we find that $\mathbf{f} = \{f^{(j)}\}_{j=1}^\infty$, where
\[f^{(j)} = \begin{cases}
\tau_{(0,\sqrt{2})}, &\text{if}\:\:j=1\\
\tau_{(0,3^{-1})}, &\text{if}\:\:1<j\le 3\\
\tau_{(0,3^{-k-1})}, &\text{if}\:\:3^{k}<j\le 3^{k+1}\:\:\text{for}\:\:k\in\N
\end{cases}.\]

\subsection{Distorted Sierpinksi Carpet} \label{ssec:Sierpinski-carpet}
Consider another IFS on $\R^2$, $\Psi = \{\psi_1, \psi_2, \psi_3, \psi_4\}$, where
\[\begin{matrix}
    \psi_1(x,y) = \frac{1}{3}I_2(x,y)^T,\hspace{1.6cm} &\psi_{2}(x,y)=\frac{1}{3}I_2[(x,y)+(2,0)]^T,\\
    \psi_3(x,y) = \frac{1}{3}I_2[(x,y)+(0,1)]^T, & \psi_4(x,y) = \frac{1}{3}I_2[(x,y)+(2,1)]^T.
    \end{matrix}\]
$A$, the invariant set of $\Psi$, turns out to be $\mathcal{C}\times\frac{1}{2}\mathcal{C}$.  Clearly, $\Psi$ does not have well-separated images since 
\[\min_{1\le j\not= k\le 4}d(\psi_j(A),\psi_k(A)) = \frac{1}{6} <\frac{\sqrt{10}}{6} = \diam \psi_j(A).\] 
Even though we cannot apply Theorem~\ref{thm : Op_WSI_IFS} directly, we can derive the formula for $\mathbf{f}$, the persistence landscape of $H_0\Cech(A)$, using a similar argument to what we had in the previous section.  We claim that $\mathcal{L}:L^\infty(\N\times\R)\to L^\infty(\N\times\R)$, defined by $\mathcal{L}\mathbf{g} = \mathbf{h}$, where 
\[\begin{split}
h^{(1)} &= \tau_{\big(0,\tfrac{\sqrt{5}}{2}\big)},\:\: h^{(2)} = \tau_{(0,\tfrac{1}{3})}\\
h^{(3)} &= h^{(4)} = \tau_{(0,\tfrac{1}{6})}\\
h^{(4k+1)}(t) &= h^{(4k+2)}(t) = h^{(4k+3)}(t) = h^{(4k+4)}(t) = \frac{1}{3}g^{(k+1)}(3t)\:\:\text{for}\:\:k\in\mathbb{N}.
\end{split},\]
satisfies $\mathcal{L}\mathbf{f} = \mathbf{f}$.

In order to prove this, we take $S_0 = \{(0,0),(1,0),(0,\frac{1}{2}), (1,\frac{1}{2})\}$, and define a sequence of sets $\{S_n\}_{n=1}^\infty$, where 
\[S_1 = \Psi(S_0),\:\:S_{n+1} = \Psi(S_n)\:\:\text{for}\:\:n\in\N.\]
Let $\mathbf{f}_n$ denote the persistence landscape function for $S_n$, then we claim that $\mathcal{L}$ satisfies $\mathcal{L}\mathbf{f}_n = \mathbf{f}_{n+1}$ for all $n\in\N$.  Because $\lim_{n\to\infty}d_H(S_n,A) =0$, it follows that $\lim_{n\to\infty}\mathbf{f}_n = \mathbf{f}$, and as we have reasoned in previous examples, this implies that $\mathcal{L}\mathbf{f} = \mathbf{f}$.

To prove that $\mathcal{L}\mathbf{f}_n = \mathbf{f}_{n+1}$, fix $n\in\N$.  Since $\diam(A) = \frac{\sqrt{5}}{2}$, it is clear that $f_{n+1}^1 = \tau_{\big(0,\frac{\sqrt{5}}{2}\big)}$.  For $s\le t$, let $M^{s,t}:H_0\Cech(S_n,s)\to H_0\Cech(S_n,t)$ and $P^{s,t}:H_0\Cech(S_{n+1},s)\to H_0\Cech(S_{n+1},t)$ denote the homomorphisms induced by the obvious inclusion
mappings. 
Since $\Cech(S_{n+1},\varepsilon)$ is path connected for $\varepsilon\ge \frac{1}{3}$, this implies that $\rank H_0\Cech(S_{n+1},\varepsilon) = 1$.  On the other hand, for $\varepsilon \in[\frac{1}{6},\frac{1}{3})$, we know that $\rank H_0\Cech(S_{n+1},\varepsilon)= 2$ since $\Cech(S_{n+1},\varepsilon)$ is not path connected, with $\psi_1(S_n)\cup\psi_3(S_n),\psi_2(S_n)\cup\psi_4(S_n)\subset S_{n+1}$ each forming an $\varepsilon$-component of $S_{n+1}$.  For $\varepsilon <\frac{1}{6}$, we know that $\rank H_0\Cech(S_{n+1},\varepsilon)\ge 4$ since 
\[\min_{1\le j\le 4}d(\psi_j(S_n),S_{n+1}\backslash\psi_j(S_n)) = \frac{1}{6}.\]
Therefore, for $t\in\R$, 
\[f_{n+1}^{(2)}(t) = \sup\{m\ge 0 \rank P^{t-m,t+m}\ge 2\} = \max\{t,\tfrac{1}{3}-t\} = \tau_{(0,\tfrac{1}{3})}(t).\]
and for $j = 3,4$
\[f_{n+1}^{(j)}(t) = \sup\{m\ge 0 \rank P^{t-m,t+m}\ge 4\} = \max\{t,\tfrac{1}{6}-t\} = \tau_{(0,\tfrac{1}{6})}(t).\]

To compute $f_{n+1}^{(j)}$ for $j\ge 5$, the key observation is that for $0\le s\le t<\frac{1}{6}$,
\begin{equation}\label{eq : rankPst_distortedC}
    \rank P^{s,t} = 4\rank M^{3s,3t},
\end{equation}
and for $s<0$ or $t\ge {1}{6}$, we get $\rank P^{s,t}\le 2$.  We have already established the second part.  To justify Equation~\eqref{eq : rankPst_distortedC}, since $\varepsilon < \frac{1}{6}$ implies that
\[\min_{1\le j\not=k\le 4}d(\psi_j(S_n),\psi_k(S_n)) >\varepsilon,\]
it follows from Lemma~\ref{lemma : dir_sum_hom} that
\[H_0\Cech(S_{n+1},\varepsilon)\cong\bigoplus_{j=1}^4 H_0\Cech(S_n,3\varepsilon).\]
As in the previous example, 
this isomorphism implies $\eqref{eq : rankPst_distortedC}$.  Using identical reasoning as in the previous section, this implies that for $j = 4k+l$ for $k\in\N$ and $l\in\{1,2,3,4\}$,
\[f_{n+1}^{(j)}(t) = \frac{1}{3}f_n{(k+1)}(3t).\]
Thus $\mathcal{L}\mathbf{f}_n = \mathbf{f}_{n+1}$ for all $n\in\N$.  Therefore $\mathcal{L}\mathbf{f} = \mathbf{f}$.

We compute the formula for $\mathbf{f} = \{f^{(j)}\}_{j=1}^\infty$ by taking the limit of $\mathcal{L}^n\mathbf{0}$ as $n\to\infty$.  We find that
\[f^{(j)} = \begin{cases}
\tau_{\big(0,\frac{\sqrt{5}}{2}\big)}, &\text{if}\:\:j=1\\
\tau_{(0,3^{-1})}, &\text{if}\:\:j=2\\
\tau_{(0,6^{-1})}, &\text{if}\:\:2<j\le 4\\
\tau_{(0,3^{-k-1})}, &\text{if}\:\:4^{k}<j\le 2\cdot 4^{k}\:\:\text{for}\:\:k\in\N\\
\tau_{(0,3^{-k}6^{-1})}, &\text{if}\:\:2\cdot 4^k<j\le 4^{k+1}\:\:k\in\N
\end{cases}.\]
In terms of Equation~\eqref{eq : PL_wellseperated}, we have $N = 4$, $c = \frac{1}{3}$, $\delta_1 = \frac{\sqrt{5}}{2}$, $\delta_2 = \frac{1}{3}$, and $\delta_3=\delta_4 = \frac{1}{6}$.  This means the persistence landscape of $H_0\Cech(A)$ is consistent with the formula from Theorem~\ref{thm : Op_WSI_IFS} even though $\Psi$ does not have well-separated images.  

\subsection{Remarks}

The 1/5 Cantor set in \ref{ssec : mod1/5} demonstrates how a reasonable formula for the persistence landscape of $H_0\Cech(A)$ can be found in the one-dimensional case, even when images are not well-separated. 
As long as the C\v ech complex of the intersection of these images is reasonable, we can use the Mayer-Vietoris sequence to make precise the difference between $\rank H_0\Cech(\Psi(S_n),\varepsilon)$ and what the rank would be if $\Psi$ had well-separated images.  The example also shows the intuition that each time two images are touching, we have to ``skip" one of the first $N$ terms in the sequence that makes up $\mathbf{f}$. 

The final two examples demonstrate that at least for zero-dimensional homology, the well-separated assumption can be too restrictive.  We might be better off replacing the well-separated hypothesis with the assumption that
\[\inf\{\varepsilon>0|\:C(A,\varepsilon)<N\}\ge \max_{1\le j\le N}\inf\{\varepsilon>0|\:C(\psi_j(A),\varepsilon)=1\} ,\]
since the proof of Theorem~\ref{thm : Op_WSI_IFS} seems to work as long each image becomes path connected by the time any two different images become path connected.
However, the well-separated assumption might be necessary for finding the persistence landscape for $H_p\Cech(A)$ for $p\ge 1$.  For $p = 1$, it would guarantee that there are no loops persisting in individual images by the time loops could be ``born" consisting of 1-chains consisting from a combination of different images.




\textbf{Acknowledgments.} Michael J. Catanzaro, Lee Przybylski, and Eric S. Weber were supported in part by the National Science Foundation under award \#1934884.  Lee Przybylski and Eric S. Weber were supported in part by the National Science Foundation and the National Geospatial Intelligence Agency under award \#1830254.

\bibliographystyle{amsplain}
\bibliography{mybib,fractal,tda}

\end{document}